\documentclass[10pt,a4paper]{amsart}
\usepackage{amsmath,amsthm,amssymb}
\usepackage{times}
\numberwithin{equation}{section}
\usepackage{bbm}
\usepackage{enumitem}
\usepackage[left=1.15in,right=1.15in,top=1.15in,bottom=1.15in]{geometry} 
\usepackage[utf8]{inputenc}
\usepackage{textcomp}
\usepackage{calc}
\usepackage{ifthen,tabularx,graphicx,multirow}
\usepackage[table]{xcolor}
\usepackage{rotating}
\usepackage{stackengine}
\usepackage{MnSymbol}
\usepackage{mathtools}
\usepackage{subfig}
\newcolumntype{Y}{>{\centering\arraybackslash}X}
\usepackage{hyperref}
\usepackage{cite}

\newtheorem{theorem}{Theorem}
\newtheorem{lemma}[theorem]{Lemma}
\newtheorem{corollary}[theorem]{Corollary}
\newtheorem{definition}[theorem]{Definition}
\theoremstyle{remark}

\numberwithin{theorem}{section}

\newcommand{\wal}{\operatorname{wal}}

\newcommand{\Wal}{\operatorname{Wal}}

\usepackage{scalerel}
\usepackage{stackengine}
\stackMath
\def\hatgap{2pt}
\def\subdown{-2pt}
\newcommand\reallywidehat[2][]{%
	\renewcommand\stackalignment{l}%
	\stackon[\hatgap]{#2}{%
		\stretchto{%
			\scalerel*[\widthof{$#2$}]{\kern-.6pt\bigwedge\kern-.6pt}%
			{\rule[-\textheight/2]{1ex}{\textheight}}
		}{0.5ex}
		_{\smash{\belowbaseline[\subdown]{\scriptscriptstyle#1}}}%
}}

\mathtoolsset{showonlyrefs}

\begin{document}

\title[]{On the Stable Sampling rate for binary measurements and wavelet reconstruction}

\author{A. C. Hansen} 
\address{DAMTP, University of Cambridge}
\email{ach70@cam.ac.uk}

\author{L. Thesing} 
\address{DAMTP, University of Cambridge}
\email{lt420@cam.ac.uk}

\begin{abstract}
This paper is concerned with the problem of reconstructing an infinite-dimensional signal from a limited number of linear measurements. In particular, we show that for binary measurements (modelled with Walsh functions and Hadamard matrices) and wavelet reconstruction the stable sampling rate is linear. This implies that binary measurements are as efficient as Fourier samples when using wavelets as the reconstruction space. Powerful techniques for reconstructions include generalized sampling and its compressed versions, as well as recent methods based on data assimilation. Common to these methods is that the reconstruction quality depends highly on the subspace angle between the sampling and the reconstruction space, which is dictated by the stable sampling rate. As a result of the theory provided in this paper, these methods can now easily use binary measurements and wavelet reconstruction bases.
\end{abstract}

\keywords{Sampling theory, Generalized sampling, Wavelets, Walsh functions, Stable Sampling rate, Data assimilation, Hilbert spaces}

\maketitle

\section{Introduction}

Reconstructing infinite-dimensional signals from a limited number of linear measurements is a key problem in sampling and approximation theory, and has received substantial attention over the last decades due to its many applications. The list of fields is comprehensive and includes Magnetic Resonance Imaging (MRI)  \cite{MRI, Unser_MRI}, electron tomography \cite{lawrence2012et,leary2013etcs}, lensless cameras, fluorescence microscopy \cite{Candes_PNAS, Roman}, X-ray computed tomography \cite{Stanford_CT, quinto2006xrayradon}, surface scattering \cite{Nature_sci_rep} among others. Efficient methods for such problems date back to Shannon's sampling theorem \cite{ShannonOverview, Shannon, Shannon50}, however, over the years, more modern approaches have been developed. Indeed, new methods include generalized sampling, which has been studied by Adcock, Hansen, Hrycak, Gr{\"o}chenig, Kutyniok, Ma, Poon, Shadrin and others \cite{2DCase, GS, sharpBounds, hrycakIPRM, Hansen_JAMS,AHPWavelet, Ma}, its compressed versions investigated by Adcock, Hansen, Kutyniok, Lim, Poon and Roman \cite{breaking, KutyniokLimShearletFourier, BAACHGSCS, PoonFrames} as well as the predecessor; consistent sampling, analysed by Aldroubi, Eldar, Unser and others \cite{SamplingTranslates, ConsistentSampling, RobustConsistentSampling, eldar2003FAA, RobustConsistentSampling3, RobustConsistentSampling4}. Note that consistent sampling is very much related to the finite section method \cite{FiniteSection1, FiniteSectionGroechnig, FiniteSection3, FiniteSection4}. More recently, new methods based on data assimilation have been successfully developed and analysed. A first approach for the same number of measurements $M$ and reconstructed coefficients $N$ was introduced under the name \textit{generalized empirical interpolation method} by Maday, Patera, Penn and Yano in \cite{maday2}. This was then further extended to \emph{Parametrized-Background Data-Weak (PBDW) approach} with $M \geq N$ in \cite{PBDW, maday3}. The PBDW approach was additionally analysed and shown to be optimal by Binev, Cohen, Dahmen, DeVore, Petrova, and Wojtaszczyk \cite{deVore1, deVore2, deVore3}.

The problem is given as follows. An element $f \in \mathcal{H}$, where $\mathcal{H}$ is a separable Hilbert space, is to be reconstructed from measurements with linear functionals  $(m_i)_{i \in \mathbb{N}} : \mathcal{H} \rightarrow \mathbb{C}$ that can be represented by elements $s_i \in \mathcal{H}$ as $m_i(f) = \langle f, s_i\rangle$. The key issue is that the $m_i$ cannot be chosen freely, but are dictated by the modality of the sampling device, for example a Magnetic Resonance Imaging (MRI) scanner providing Fourier samples or a fluorescence microscope giving binary measurements. The goal is to reconstruct $f$ from the finite number of samples $\{m_i(f)\}_{i=1}^M$ for some $M \in \mathbb{N}$. The space of the functions $s_i$ is called the \emph{sampling space} and is denoted by $\mathcal{S} = \overline{\operatorname{span}} \{ s_i : i \in \mathbb{N} \}$, meaning the closure of the span. In practice, one can only acquire a finite number of samples. Therefore, we denote by $\mathcal{S}_M = \operatorname{span} \{ s_i : i = 1, \ldots, M \}$ the sampling space of the first $M$ elements. 
The reconstruction is typically done via a reconstruction space denoted by $\mathcal{R}$ and spanned by reconstruction functions $(r_i)_{i \in \mathbb{N}}$, i.e. $\mathcal{R} = \overline{\operatorname{span}} \{ r_i : i \in \mathbb{N} \}$. As in the case of the sampling space, it is impossible to acquire and save an infinite number of reconstruction coefficients. Hence, one has to restrict to a finite reconstruction space, which is denoted by $\mathcal{R}_N = \operatorname{span} \{ r_i : i = 1 , \ldots , N \}$.
The key is that the $r_i$ can be tailored to the desired signal type to be reconstructed.  For example, spaces spanned by X-lets (wavelets, curvelets, contourlets, shearlets) \cite{Shearlets2, Shearlets3, OptimalShearlets, Curvelets, Curvelets3, Curvelets2, Contourlets2, Contourlets}
may be preferable as reconstruction spaces in imaging applications, whereas polynomials may be useful when recovering very smooth functions.

The methods mentioned above can be described as follows: for $f \in \mathcal{H}$ and $N,M \in \mathbb{N}$, we define the reconstruction method of \emph{generalized sampling} $G_{N,M} : \mathcal{H} \rightarrow \mathcal{R}_N$ by
	\begin{equation}\label{EqGeneralSamp}
	\langle P_{\mathcal{S}_M} G_{N,M}(f), r_j \rangle = \langle P_{\mathcal{S}_M}f, r_j \rangle, \quad r_j \in \mathcal{R}_N,
	\end{equation}
	where $P_{\mathcal{S}_M}$ denotes the orthogonal projection on the subspace $\mathcal{S}_M$.  Note that the stability and accuracy of this method depends on the subspace angle between the sampling and the reconstruction space, i.e.
	\begin{equation}
		\| f - G_{N,M}(f)\| \leq \mu(\mathcal{R}_N,\mathcal{S}_M) \| f - P_{\mathcal{R}_N} f\|,
	\end{equation}
	where we define the subspace angle between closed subspaces $U,V \in \mathcal{H}$ 
	\begin{equation}
\cos(\omega(U,V)) := \frac 1 {\mu(U,V)} := \inf\limits_{u\in U, \|u\| =1} \| P_V u\|, 
\end{equation}
$\omega(U,V) \in [0, \pi/2]$.
Moreover, the condition number $\kappa$ of $G_{N,M}$ is also given by $\kappa (G_{N,M}) = \mu(\mathcal{R}_N,\mathcal{S}_M)$.
For the PBDW method one calculates
	\begin{equation}
		F_{N,M}(f) = {\displaystyle {\underset {u \in P_{\mathcal{S}_M}f + \mathcal{S}_M^\perp}{\operatorname {argmin} }}}\,  \|u - P_{\mathcal{R}_N} u \|,
	\end{equation}
and it can be shown that the accuracy then depends on subspace angle as follows
	\begin{equation}
	\| f - F_{N,M}(f)\| \leq \mu(\mathcal{R}_N,\mathcal{S}_M) \operatorname{dist}(f,\mathcal{R}_N \oplus (\mathcal{S}_M \cap \mathcal{R}_N)^{\perp}).
	\end{equation}
	Moreover, this is sharp because the constant $\mu(\mathcal{R}_N,\mathcal{S}_M)$ cannot be improved.
It is clear that, in both approaches, the key to success lies in the ability to make sure that 
$$
\mu(\mathcal{R}_N,\mathcal{S}_M)  \leq \theta, \qquad \theta \in (1,\infty).
$$
Thus, we need to balance the number of samples $M$ with the number of reconstruction vectors $N$, and this leads to the so-called \emph{stable sampling rate}:
\begin{equation}\label{ss_rate}
\Theta(N, \theta) = \min \left\{ M \in \mathbb{N}: \mu(\mathcal{R}_N,\mathcal{S}_M)  \leq \theta \right\}.
\end{equation}
The methods above can only be used efficiently when the stable sampling rate is known and reasonable. In particular, numerical calculations of the stable sampling rate are very time consuming. Moreover, if the stable sampling rate is worse than linear, the approximation quality of the reconstruction space must allow for rapid approximation to compensate for the "slow" sampling rate. Fortunately, it is possible to obtain sharp results on describing $\Theta(N, \theta)$ for popular sampling and reconstruction spaces, and often, especially for the reconstruction with X-lets, one can establish linearity. In this paper we do so for sampling with Walsh functions and reconstructing with wavelets. Remark that this is not always the case. For example in the case of the reconstruction with Legendre polynomials, the stable sampling rate is known to be quadratic for Fourier measurements \cite{hrycakIPRM}. Additionally, for the PBDW-method there are approaches where the reconstruction space is fixed by the PDE and the sampling space is chosen adaptively. This is a very different setting from the one discussed here. Details and convergence rates can be found in \cite{Binev4}.

\subsection{Connection to previous work and novelty of the paper}
The stable sampling rate is well understood when the samples $m_i(f) = \langle f, s_i\rangle$ are Fourier measurements. In other words, the $s_i$ are complex exponentials and $m_i(f)$ are the Fourier coefficients. In this case the stable sampling rate is linear for many X-lets including wavelets and shearlets. Fourier samples and X-lets are a natural starting point given the vast applications that are based on Fourier measurements (MRI, tomography problems with parallel beam, surface scattering, radio interferometry etc.). However, the next question regards binary measurements. By binary measurements we mean that the sampling functions $s_i$ can only take two values either $\{0,1\}$ or $\{-1,1\}$. Without loss of generality we can assume that the model uses $\{-1,1\}$, as one can, by adding one extra measurement with the constant function, convert from the $\{0,1\}$ setup to the $\{-1,1\}$ model.

Binary measurements are a mainstay in signal and image processing due to the "on-off" nature of many physical sampling devices. Microscopy is an obvious application as well as the newly emerging techniques of lensless cameras. In the discrete setting binary measurements are often modelled with Hadamard matrices, and this is one of the reasons why Hadamard matrices are so important in signal processing. To model binary measurements we change the model from Fourier samples $\langle f, s_i\rangle$, where the $s_i$ are complex exponentials to letting the $s_i$ be Walsh functions. The Walsh functions are the binary counterpart to Fourier samples and complex exponentials. Thus, the key question is as follows: what is the stable sampling rate when sampling with Walsh functions and reconstructing with wavelets? The answer is that it is linear regardless of the dimension when we consider separable boundary wavelets. This means that sampling with binary measurements is as efficient (up to potentially a different constant) as sampling with Fourier samples when reconstructing with wavelets. We expect the techniques used in this paper to extend to other X-lets as well; however, the extension, as in the Fourier case, is non-trivial.

\subsection{Main Theorem}
We consider the sampling space $\mathcal{S}$ of Walsh functions, which will be described in more detail in Chapter \ref{Ch:Walsh}, and let the reconstruction space $\mathcal{R}$ be the space of boundary-corrected Daubechies wavelets (see Chapter \ref{Ch:Wavelet} for details). The main theorem states that the stable sampling rate is indeed linear in $N$.
\begin{theorem}\label{Theo41CH2}
		Let $\mathcal{S}$ and $\mathcal{R}$ be the sampling and reconstruction space spanned by the d-dimensional Walsh functions and separable boundary wavelets respectively. Moreover, let $N = 2^{dR}$ with $R \in \mathbb{N}$. Then for all $\theta \in (1, \infty)$ there exists $S_\theta$ such that for all $M \geq 2^{dR}S_\theta$ we have $\mu(\mathcal{R}_N,\mathcal{S}_M) \leq \theta$. In particular, one gets $\Theta \leq S_\theta N$. Hence, the relation $\Theta(N;\theta) = \mathcal{O}(N)$ holds for all $\theta \in (1,\infty)$.
	\end{theorem}

\section{Walsh functions - Defining the Sampling Space $\mathcal{S}_M$}\label{Ch:Walsh}

	Due to the fact that we are dealing with the $d$-dimensional case, we introduce multi-indices to make the notation more readable. Let $j = \left\{j_1, \ldots, j_d \right\} \in \mathbb{N}^d$, $d \in \mathbb{N}$ be a multi-index. A natural number $n$ is in the context of a multi-index interpreted as a multi-index with the same entry, i.e. $n = \left\{n, \ldots ,n \right\}$. Then we define the addition of two multi-indices for $j, r \in \mathbb{N}^d$ by the pointwise addition, i.e. $j + r = \left\{j_1 +r_1, \ldots, j_d +r_d \right\}$ and the sum 
\begin{equation}\label{sum}
\sum_{j=k}^{r} x_j := \sum_{j_1 = k_1}^{r_1} \ldots \sum_{j_d = k_d}^{r_d} x_{j_1, \ldots, j_d}, 
\end{equation}
where $k,r \in \mathbb{N}^d$. The multiplication of an multi-index with a real number is understood pointwise, as well as the division by a multi-index. The absolute value of a multi-index $j$ is given by $|j|= j_1 + \ldots +j_d$. The $d$ dimensional functions that we use in this paper are constructed by the tensor product. For a function $f: \mathbb{R} \rightarrow \mathbb{R}$ and an input parameter $\left\{x_i\right\}_{i=1,\ldots,d} = x \in \mathbb{R}^d$ with $x_i \in \mathbb{R}$ we use the following notation to present the $d$-times tensor product of $f$, i.e.
\begin{equation}
	f(x) = f(x_1) \otimes \ldots \otimes f(x_d) \text{ (d-times)}.
\end{equation}
It should be clear from the input parameter, whether $f$ represents the function on $\mathbb{R}$ or $\mathbb{R}^d$.

	\subsection{Defining Walsh functions}
The key property that makes Walsh functions attractive in many applications is that they take only the values $1$ and $-1$. 	However, as Walsh functions are defined in the dyadic analysis, some properties only hold for dyadic addition. Recalling the basics of dyadic addition,
	we represent elements $x \in \mathbb{R}_+$ with their dyadic representation as follows
	$$
		x = \sum\limits_{i \in \mathbb{Z}} x_i 2^i ,
	$$
	where $x_i \in \left\{0,1\right\}$ for all $i \in \mathbb{Z}$. The natural extension always ends in $0$ for dyadic rational numbers and is infinite for dyadic irrational numbers. The representation is therefore unique. Elements of $\mathbb{R}_{-}$ are represented as in the decimal analysis with an additional $-$ in front of the representation. In the dyadic analysis the addition $\oplus: \mathbb{R}_+ \mapsto \mathbb{R}_+$ is defined by
	\begin{equation}
	x \oplus y = \sum_{i \in \mathbb{Z}} (x_i \oplus_2 y_i) 2^i,
	\end{equation}
	where $x_i \oplus_2 y_i$ is addition modulo two, i.e. $0 \oplus_2 0  = 0, 0 \oplus_2 1 = 1, 1 \oplus_2 0 = 1, 1 \oplus_2 1 = 0$. For negative numbers one has $-x \oplus y = x \oplus -y = - (x \oplus y)$.
	Note that there is no closed form between the decimal and dyadic addition. In particular, for two numbers $x,h \in \mathbb{R}$ the expression of the decimal sum $x + h$ has a different expression in the dyadic addition for every pair of numbers. This leads to further investigation during the proof of the main theorem. One part on the way to control the subspace angle is not to deal with all timeshifts of the wavelet but instead transfer them to the Walsh function and deal with the Walsh polynomial. Unfortunately, all properties of Walsh functions rely on dyadic rather than decimal addition. Therefore, this difference of the additions needs special care. In chapter \ref{Ch:LemmaWalsh} we will see that under mild assumptions the additions can made be equal. In chapter \ref{Ch:WaveletChanges} we tweak the wavelets to match these assumptions.
	
	Now, we present the Walsh functions, which are used to represent the sampling space $\mathcal{S}_M$. Therefore, we give a definition of the classical Walsh functions that highlights the difference between the possible orderings.
\begin{definition}[\cite{deutschWalsh}]
		Let $s \in \mathbb{N}$ and $x \in [0,1)$. Then there exists a unique $n = n(s) \in \mathbb{N}$ such that $s = \sum_{i=0}^{n-1} s_i 2^i$, in particular $s_{n-1} \neq 0$ and $s_k = 0$ for all $k \geq n$. Let $s^n = \left\{s_0, \ldots, s_{n-1} \right\}$ and for $x = \sum_{i = - \infty}^{-1} x_i 2^{i}$ define $x^n = \left\{x_{-n}, \ldots, x_{-1} \right\}$, and $\omega_W$: $\mathbb{R}^n \mapsto \mathbb{R}^n$ by
		\begin{align}
		\omega_W = \begin{pmatrix}
		    0  & \cdots & 0 & 1 & 1\\
		  \vdots  & \udots & \udots & 1 & 0 \\
		  0  & \udots  & \udots & \udots & \vdots  \\
		  1 & 1 & \udots &  & \vdots \\
		  1 & 0 & \cdots & \cdots & 0 
				 \end{pmatrix}.
		\end{align}
The \emph{Walsh functions} are then given by
		\begin{equation}
		\wal(s;x) = (-1)^{s^n \cdot \omega_W x^n}.
		\end{equation}
	\end{definition}
	
By changing the matrix $\omega_W$ one gets different orderings of the Walsh functions. For example, the identity matrix leads to the Walsh-Kronecker functions, which have the drawback that with a change of $n(s)$ all functions are altered, hence one has to fix the maximal $s$ in advance. The Walsh-Paley ordering is obtained by replacing $\omega_W$ with the reversal matrix, i.e. 
	\begin{align}
\omega_{WP} = \begin{pmatrix}
0  & \cdots & 0 & 0 & 1\\
\vdots  & \udots & \udots & 1 & 0 \\
0  & \udots  & \udots & \udots & 0  \\
0 & 1 & \udots & \udots & \vdots \\
1 & 0 & 0 & \cdots & 0 
\end{pmatrix}.
\end{align}
They overcome the previous problem, but the functions are not ordered such that the number of zero crossings increases with $s$. Both drawbacks are overcome with the Walsh-Kaczmarz ordering presented in the previous definition.

	The classical Walsh functions can be extended to the \emph{generalized Walsh functions} $\Wal: \mathbb{R}^2_+ \rightarrow \left\{-1,1\right\}$ which are defined with the classical Walsh functions and the periodic continuation with period $1$ by
	\begin{equation}
	   \operatorname{Wal}(s,x) = (-1)^{s_0 x_0} \operatorname{wal}(\left[s\right]; x)\operatorname{wal}(\left[x\right]; s),
	\end{equation}
	where $s$ and $x$ have the dyadic representation $(s_i)_{i \in \mathbb{Z}} $ and $(x_i)_{i \in \mathbb{Z}}$ and $s_0$, $x_0$ are the corresponding elements of the sequence. This extension can also be defined by letting $\omega_W$ be infinite, i.e. be defined over $\mathbb{Z}$ instead of $\mathbb{N}$ and hence allow inputs with infinite dyadic representations over $\mathbb{Z}$.
	Moreover, the Walsh functions can also be extended to negative inputs. Therefore, we define  the following equality as in \cite{deutschWalsh}
	\begin{align}\label{negativeSequ}
	\Wal(-s,x) :=  - \Wal(s,x)\\
	\Wal(s,-x) := - \Wal(s,x).
	\end{align}
	The Walsh functions in higher dimensions are obtained by the tensor product, i.e.
	for $s = \left\{s_k\right\}_{k=1,\ldots,d}, x = \left\{x_k\right\}_{k=1,\ldots,d} \in \mathbb{R}^d$
	\begin{equation}
	\operatorname{Wal}(s,x) = \bigotimes_{k=1}^d  \operatorname{Wal}(s_k,x_k) .
	\end{equation}
	The Walsh functions can also be combined to \emph{Walsh polynomials} similar to trigonometric polynomials.
	\begin{definition}
		Let $A,B \in \mathbb{Z}^d$ such that $A_i \leq B_i, i =1, \ldots,d$ and $\alpha_{j_i} \in \mathbb{R}$. Then for $z \in \mathbb{R}^d_+$ we define the \emph{Walsh polynomial} of order $n = |B|$ by $\Phi(z) = \sum_{j=A}^{B} \alpha_{j} \Wal(j,z)$. The set of all Walsh polynomials up to degree $n$ is given by
		\begin{equation}
			WP_n = \left\{\sum_{j=A}^{B} \alpha_{j} \Wal(j,z),\alpha_{j_i} \in \mathbb{R},  A,B \in \mathbb{Z}^d , |B| \leq n \right\}.
		\end{equation}
	\end{definition}
	
	With the generalized Walsh functions one can define a continuous and discrete transform.
	To ensure that the following integral exists, let $f \in L^2([0,1]^d)$, the \emph{generalized Walsh transform} is given almost everywhere by
 \begin{equation}\label{GenerWalsTransEq}
	 	\reallywidehat[W]{f}(s) = \langle f(\cdot), \Wal(s,\cdot) \rangle = \int_{[0,1]^d} f(x)\Wal(s,x)dx, \quad s \in \mathbb{R}^d.
	 \end{equation}
	This is suitable for our setting, because we consider only the Walsh transform of functions that are supported in $[0,1]^d$. 
	In the discrete setting we have for $N = 2^n$, $n \in \mathbb{N}$ and $x = \left\{x_0,\ldots,x_{N-1}\right\} \in \mathbb{R}^N$ that the one dimensional \emph{discrete Walsh transform} of $x$ is given by $X = \left\{X_0,\ldots,X_{N-1}\right\}$ with
	\begin{equation}
	X_j = \frac 1 N \sum\limits_{k=0}^{N-1} x_k \Wal(j,\frac k N).
	\end{equation}
	This transform corresponds, as mentioned, to the multiplication with a Hadamard matrix. By the definition of $\Wal$ it corresponds to the Hadamard matrix in Walsh-Kaczmarz ordering. The addition here is again the dyadic addition. The discrete d-dimensional Walsh transformed of $x \in \mathbb{R}^{N_1 \times \ldots \times N_d}$ where $x_{k_i} \in \mathbb{R}$, $k=\left\{k_i\right\}_{i=1,\ldots,d}, k_i=0,\ldots,N_i-1 $ is given by $X = \left\{X_{j}\right\} \in \mathbb{R}^{N_1 \times \ldots \times N_d}$, where $X_{j_i} \in \mathbb{R}$, $j = \left\{j_i\right\}_{i=1,\ldots,d}, j_i=0,\ldots,N_i-1$, with
	\begin{equation}
	X_{j} = \frac 1 {\prod_{i=1}^{d} N_i} \sum\limits_{k=0}^{N-1} x_{k} \Wal(j,\frac k N).
	\end{equation}

	\subsection{Properties of Walsh functions}
	
	The Walsh functions obey the following properties: They are symmetric, 
	\begin{equation}
		\operatorname{Wal}(s,x) = \operatorname{Wal}(x,s) \text{ for all } s,x \in \mathbb{R},
	\end{equation}
	  and they obey the \emph{scaling property} as well as the \emph{multiplicative identity}, i.e 
	  \begin{equation}\label{Eq:scalingProperty}
	  \Wal(2^k s, x) = \Wal(s, 2 ^k x) \text{ for all } s,x \in \mathbb{R},~ k \in \mathbb{N}
	   \end{equation}
	   and 
	   \begin{equation}\label{Eq:MultiplicativeID}
	   	\Wal(s,x)\Wal(s,t) = \Wal(s,x \oplus t) \text{ for all } s,x \in \mathbb{R}.
	   \end{equation}
	   These properties can be directly transferred to the $d$-dimensional Walsh functions and the continuous Walsh transform, i.e. it holds, that the continuous Walsh transform is linear 
	   \begin{equation}
	   	\mathcal{W} \left\{ a f(t) + b g(t) \right\} = a \mathcal{W} \left\{ f(t) \right\} + b \mathcal{W} \left\{ g(t) \right\} \text{ for all } a, b \in \mathbb{R} \text{ and } f,g \in L^2 ([0,1]^d),
	   \end{equation}
	   and obeys the following \emph{shift} and \emph{scaling property}, i.e. 
	   \begin{equation}
	   	\mathcal{W}\left\{ f(t \oplus x) \right\}(s) = \mathcal{W} \left\{ f(t) \right\}(s) \Wal(x,s) \text{ for all } x \in \mathbb{R}^d \text{ and }f \in L^2 ([0,1]^d)
	   \end{equation}
	    and 
	    \begin{equation}
	    \mathcal{W}\left\{ f(2^m t) \right\} (s) = \frac 1 {2^m} \mathcal{W} \left\{ f(t) \right\} (\frac s {2^m}) \text{ for all } m \in \mathbb{N}^{d} \text{ and } f \in L^2 ([0,1]^d).
	    \end{equation}

\section{Wavelets - Defining the Reconstruction Space $\mathcal{R}_N$}\label{Ch:Wavelet}
	
	\subsection{Boundary Wavelets}
	
	\subsubsection{Boundary Wavelet space in one dimension}

		Daubechies boundary wavelets are deduced from general Daubechies wavelets. They have the advantage that they keep desirable properties, such as smoothness and vanishing moments, from their mother wavelet. In contrast, other approaches to find an orthonormal wavelet basis for $L^2([0,1])$, such as extension with zero, periodising or folding, loose smoothness. For the construction of the Daubechies boundary wavelets, as presented in \cite{boundaryWavelets}, one starts with the Daubechies scaling functions. First, we deal with the scaling functions on the positive line $[0,\infty)$. Remember that a Daubechies scaling function $\phi$ of order $p$ has the support $[-p+1,p]$. Then the functions $\phi_n (x) = \phi(x-n)$ have their support completely in $[0,\infty)$ for $n \geq p-1$. However, these functions do not even generate the polynomials on $[0,\infty)$, so they do not represent smooth functions well. Therefore, the following functions are added to circumvent this issue:
	\begin{equation}
	\widetilde{\phi}^{\text{left}}_n(x) = \sum\limits_{l = 0}^{2p-2}  \binom{l}{n} \phi ( x + l - p +1) .
	\end{equation}
	It is shown in \cite{boundaryWavelets} that these functions together with the inner ones, i.e. the translates of the scaling function, whose support is completely contained in the positive real line, span all polynomials with degree smaller or equal to $p-1$ on $[0,\infty)$. 
	The same construction can be done for the negative line $(-\infty,0]$ and then be shifted by $1$ to get to the desired interval. This means in detail that the scaling function on the right hand side can be deduced from those on the left side, i.e. the construction for the right hand side results from a shift of $1$ in the functions that do intersect with the right end of the interval and a reflection. We have that 
	\begin{equation}
	\widetilde{\phi}_n^{\text{right}}(x) = \widetilde{\phi}_{-1-n}^{\text{left}}(-x) .
	\end{equation}
	In the next step we bring both systems together on $[0,1]$. To make sure that each shift of the scaling function is either an inner a left or a right scaling function, we consider scaling functions at a level $j \geq  J_0$, where $2^{J_0} \geq 2p -1$. This way the support size of the scaling function at that scale is smaller than $1$. Therefore, the scaling function can intersect only with $0$ or $1$ and hence the correction is well defined.
	The functions are now all corrected on the boundaries and they span the desired space $L^2([0,1])$. To form an orthonormal basis, we simply apply a Gram-Schmidt procedure. The new functions, after the orthonormalisation, are denoted by $\phi_n^\text{left}, \phi_n^\text{right}$. The functions have staggered support, i.e. $\operatorname{supp} \phi_n^\text{left} = [0,p+n]$. Therefore, all $\phi$ have support length at most $2p-1$. Hence, the change to the boundary wavelet preserves the favourable property of a small support size.
	The dilated boundary scaling functions can be deduced from this construction, like the scaling functions for the real line.
	With this construction we obtain $2^j +2$ scaling functions at every scale $j$, but in many applications one prefers to have $2^j$ scaling functions. Therefore, we remove the two outermost interior scaling functions, i.e. those with the support closest to $0$ and $1$ but not intersecting with them.
	This results in the subspaces
	\begin{equation}
	V_j^b = \operatorname{span} \left\{ \phi_{j,n}^b : n=0, \ldots 2^j-1 \right\},
	\end{equation}
	where
	\begin{align} \label{EqBoundScaling}
	\phi_{j,n}^b(x) = \begin{cases}
	2^{j/2} \phi^{\text{left}}_n (2^j x)  &  n=0,\ldots p-1 \\
	2^{j/2} \phi_n (2^j x) & n = p,\ldots 2^j - p -1 \\
	2^{j/2} \phi_{2^j -n-1} ^{\text{right}} (2^j (x-1)) & n = 2^j -p, \ldots 2^j -1 .
	\end{cases}
	\end{align}
	In \cite{boundaryWavelets} it is proven that we can define the wavelet space at every scale $j$ similar to the case on the real line by
	\begin{equation}
	W_j^b = V_{j+1}^b \cap (V_j^b)^\perp .
	\end{equation}
	The original wavelet functions $\psi_{j,k}$ from the real line are in $W_j^b$ for $k = p, \ldots , 2^j - p -1$. Because of the dimension of the scaling space one knows that $\dim W_j^b = 2^j$. Hence, one has to add $2p$ wavelets. As they do not play an important role for the investigation of the main theorem we point the interested reader to \cite{boundaryWavelets} for detailed information.
	
	Now that the necessary information about wavelets is introduced, we discuss the reconstruction space. The data is usually sparsely represented in the wavelet scheme, i.e. they only have large coefficients up to a certain scale. Therefore, the reconstruction space contains only the wavelets up to some scale $R$. Moreover, the low frequency part can be represented by the scaling space at some level $J \geq J_0$. In theory the choice of the lowest level $J$ is free. Nevertheless, it is common to use $J= J_0$. This results in the following reconstruction space.
		For $R \in \mathbb{N}$, the space of wavelets up to a scaling of $R$ is given by
		\begin{equation}\label{Eq:R_N1d}
		\mathcal{R}_N = V_J^b \oplus W_J^b \oplus \ldots \oplus W_{R-1}^b = V_R^b.
		\end{equation}
		and has $N = 2^R$ elements.
	Due to the construction the "left" scaling functions are translates of the mother scaling functions and the "right" scaling functions are  reflected translated scaling functions, denoted by $\phi^\#$.
	Therefore,
	\begin{equation}\label{Eq:ScalingSpace1dim}
	V_R^b = \operatorname{span}\left\{ \phi_{R,n} : n = 0,\ldots,2^R -p -1 , \phi_{R,n}^\# : n = 2^R - p,\ldots,2^R-1  \right\}
	\end{equation}
	and every $\varphi \in \mathcal{R}_N$ with $|| \varphi || = 1$ has the representation
	\begin{equation}\label{Eq:repvarphi1d}
	\varphi = \sum\limits_{n=0}^{2^R-p-1} \alpha_k \phi_{R,n} + \sum\limits_{n = 2^R -p}^{2^R-1} \beta_k \phi_{R,n}^\# \text{ with } \sum\limits_{n=0}^{2^R-p-1} |\alpha_n|^2 + \sum\limits_{n = 2^R -p}^{2^R-1} |\beta_n|^2= 1 .
	\end{equation}
	
	\subsubsection{Boundary wavelets in higher dimensions}
	
	In this paper, we also consider the $d$-dimensional case. For the reconstruction in $d$-dimensions we focus on separable boundary wavelets. Therefore, the $d$-dimensional wavelets can be derived from the one dimensional case by tensoring the scaling space and then studying the according wavelet space. 
	
	From \eqref{EqBoundScaling} we have the one dimensional boundary scaling function. With the tensor product we get the $d$ dimensional one, i.e. $\phi_{j,n}^d = \phi_{j,n_1}^b \otimes \ldots \otimes \phi_{j,n_d}^b$ for $n = \left\{n_1, \ldots, n_d \right\} \in \mathbb{N}^d$, $j \geq J_0$. To make this easier to read we set $\phi_{j,n} : = \phi_{j,n}^d$, as the dimension is defined by the context. Then the $d$-dimensional scaling space is given by
	\begin{equation}
		V_J^{b,d} := V_J^b \otimes \ldots \otimes V_J^b \text{ (} d \text{ times)}.
	\end{equation}
	For the purpose of constructing higher dimensional boundary wavelets we exploit the MRA structure. We have that 
	\begin{equation}
	V_j^b = V_{j-1}^b \oplus W_{j-1}^b .
	\end{equation} 
	Therefore, we can divide in higher dimensions the scaling space at one level in the scaling space and the wavelet space in the lower level 
	\begin{equation}
	V_j^{b,d} = V_j^b \otimes \ldots \otimes V_j^b = (V_{j-1}^b \oplus W_{j-1}^b) \otimes \ldots \otimes (V_{j-1}^b \oplus W_{j-1}^b) = V_{j-1}^{b,d} \oplus W_{j-1}^{b,d}.
	\end{equation}
	This way we have defined the $d$ dimensional boundary-corrected wavelet space $W_{j-1}^{b,d}$ by
	\begin{equation}
		W_{j-1}^{b,d} := (V_{j-1}^b \oplus W_{j-1}^b) \otimes \ldots \otimes (V_{j-1}^b \oplus W_{j-1}^b) \ominus V_{j-1}^{b,d}.
	\end{equation}
	Due to \eqref{Eq:R_N1d} we only have to focus on the scaling space, as the sum over the wavelet spaces can be represented by the scaling space at highest scale. Therefore, we do not explain details about the wavelets here. We have with \eqref{Eq:ScalingSpace1dim}
	\begin{equation}
	V_j^{b,d} = V_j^b \otimes \ldots \otimes V_j^b = \operatorname{span} \left\{ \phi_{j,n}^d := \phi_{j,n_1}^b \otimes \ldots \otimes \phi_{j,n_d}^b : n=\left\{n_1, \ldots, n_d\right\}, n_i = 0, \ldots 2^j -1 \right\}.
	\end{equation}
	According to the size of the one dimensional scaling space, we know that the $d$-dimensional scaling space has size $2^{dj}$. The reconstruction space for $N = 2^{dR}$ is then 
	\begin{equation}
	\mathcal{R}_N = V_R^{b,d}.
	\end{equation}
	In order to get the scaling space for the boundary wavelets in one dimension we had to reflect the scaling function for translates $k = 2^j -p, \ldots, 2^j-1$. This means that $V_{j}^{b,d}$ is spanned by the translates of $2^d$ functions, which are constructed by tensoring of the original scaling function $\phi^0 := \phi$ and the translated version $\phi^1 := \phi^\#$. 
	Define $K_0 = \left\{ 0, \ldots, 2^j-p-1 \right\}$ and $K_1 = \left\{ 2^j -p ,\ldots, 2^j -1 \right\}$, then the mapping $m : \left\{0, \ldots, 2^j -1\right\} \mapsto \left\{0,1\right\}$ is given by
	\begin{align} 
	m(n) = \begin{cases}
	0   &  n \in K_0 \\
	1   & n \in K_1 .
	\end{cases}
	\end{align}
	This allows us to represent $\varphi \in V_j^{b,d}$ with $||\varphi||=1$ by 
	\begin{equation}\label{Eq:R_Ndd}
	\varphi = \sum\limits_{n=0}^{2^j-1} \alpha_n \phi_{j,n_1}^{m(n_1)} \otimes \ldots \otimes \phi_{j,n_d}^{m(n_d)} = \sum\limits_{s \in \left\{ 0,1 \right\}^d} \sum\limits_{n \in K_s} \alpha_n \bigotimes \phi_{j,n}^s \text{ with } \sum\limits_{s \in \left\{ 0,1 \right\}^d} \sum\limits_{n \in K_s} | \alpha_n |^2 = 1,
	\end{equation}
	where
	\begin{equation}
	\bigotimes \phi_{j,n}^s = \phi_{j,n_1}^{s_1} \otimes \ldots \otimes \phi_{j,n_d}^{s_d}.
	\end{equation}
	
	So at this point we have the wavelets and the scaling functions which span the space $L^2([0,1]^d)$ and therefore the reconstruction from Walsh functions in the wavelet space is guaranteed.

	\section{The main theorems and its proof}

	Taken together, we can now prove the main result.

\subsection{Useful lemmas about Walsh functions}\label{Ch:LemmaWalsh}
	
	For the proof of theorem \ref{Theo41CH2} we have to combine the properties of Walsh functions in the dyadic analysis with the properties of the wavelets in the decimal analysis.
	Therefore, we consider under which conditions the decimal and dyadic additions are equal. This is important to combine the multiplicative identity of the Walsh functions with the translates of the wavelets.
	\begin{lemma}\label{Lem:Addxm}
		Let $x \in [0,1)$ and $m \in \mathbb{N}$ then 
		$
		x + m = x \oplus m.
		$
	\end{lemma}
	
	\begin{proof}
		The dyadic representation of $x$ is $\left\{ \ldots,0,x_{-1},x_{-2},\ldots \right\}$ and the dyadic representation of 
\[
m \text{ is } \left\{ \ldots,m_2,m_1,m_0,0,0,\ldots\right\}.
\]
Because the representations do not have non-zero elements at the same position, one achieves the following
		\begin{equation}
		x \oplus m = \sum\limits_{i= -\infty}^\infty (x_i \oplus_2 m_i) 2^i = \sum\limits_{i= -\infty}^\infty (x_i + m_i) 2^i = x + m.
		\end{equation}
	\end{proof}

Next, we look at the inverse element for the dyadic addition. This is also discussed in \cite{WalshSeriesAndTransform} and will be used in Corollary \ref{Cor:scalingProp}.

		\begin{lemma}\label{LemmaDyadicAdd}
		The dyadic sum of two numbers $x,y \in \mathbb{R}_+$ is $0$ if and only if $x = y$.
	\end{lemma}
	
	\begin{proof}
		Let $x,y \in \mathbb{R}_+$ with the dyadic representation $\left\{x_i\right\}_{i \in \mathbb{Z}}$ and $\left\{y_i\right\}_{i \in \mathbb{Z}}$. Then
		\begin{equation}
		x \oplus y = \sum\limits_{i= -\infty}^\infty (x_i \oplus_2 y_i) 2^i =0
		\end{equation}
		if and only if $x_i \oplus_2 y_i = 0$ for all $i \in \mathbb{Z}$. This is the case if and only if $x_i = y_i$ for all $i \in \mathbb{Z}$, i.e. $x = y$.
	\end{proof}

	With this the relation between the decimal addition and the multiplicative identity of the Walsh functions can be found.
	
	\begin{corollary}\label{Cor:scalingProp}
		Let $t \in \mathbb{N}$ and $x \in [0,1)$, then the following holds:
		\begin{equation}
		\mathcal{W}\left\{f(x+ t) \right\} (s) = \mathcal{W}\left\{ f(x) \right\} (s) \Wal(t,s).
		\end{equation}
	\end{corollary}
	
	\begin{proof}
		With Lemma \ref{Lem:Addxm} we have that $x \oplus t = x +t$. This allows 
		\begin{equation}
		\mathcal{W}\left\{f(x+ t) \right\} (s) = \mathcal{W}\left\{f(x\oplus t) \right\} (s) = \mathcal{W}\left\{ f(x) \right\} (s) \Wal(t,s).
		\end{equation}
	\end{proof}

Next, we analyse the sum of Walsh functions with equally distributed inputs. This will be used in Lemma \ref{PlancheralWalshLemma}.

\begin{lemma}\label{LemmaAdditWalsh}
	Let $N = 2^n$, $n \in \mathbb{N}$, then for all $s \in \mathbb{N}_0$ the following property holds:
	\begin{equation}\label{EqWalAddition}
	\sum\limits_{i=0}^{N-1} \wal(s,\frac i N) = \begin{cases}
	N& \text{ if } s =0 \\
	0& \text{ else.}
	\end{cases}
	\end{equation} 
\end{lemma}

\begin{proof}
	The first case for $s = 0$ follows directly by the definition of the Walsh function as $\wal(0;x) \equiv 1$ for all $x \in [0,1)$. For the second part we use the equal distribution of Walsh functions in intervals where it takes the values $-1$ and $1$, i.e. for $s \leq 2^m$, $m \in \mathbb{N}_0$ the Walsh function $\wal(s;x)$ takes the value $1$ on $2^{m-1}$ intervals of length $\frac 1 {2^m}$ and $-1$ on the same number of intervals of that length \cite{deutschWalsh}. As the sequence $\left\{ i/N \right\}_{i =0, \ldots N-1}$ is equally distributed on this interval, the sum equals $0$. 
\end{proof}

With this information in hand we can now prove the following lemma, which shows a relation between the values of the discrete Walsh transform and the signal itself. This will then be used in Lemma \ref{Lemma2DphiSum}. 

\begin{lemma}\label{PlancheralWalshLemma}
Let $N = \left\{N_i\right\}_{i=1,\ldots,d}$, where $N_i = 2^{n_i}, n_i \in \mathbb{N}$ and $i = 1,\ldots,d$. Let $x \in \mathbb{R}^{N_1 \times \ldots \times N_d}$, where $x = \left\{x_{k}\right\}$ and $x_{k_i} \in \mathbb{R}$, $k=\left\{k_i\right\}_{i=1,\ldots,d}, k_i=0,\ldots,N_i-1 $ be a discrete d-dimensional signal. Given the discrete d-dimensional Walsh transformed by $X = \left\{X_{j}\right\} \in \mathbb{R}^{N_1 \times \ldots \times N_d}$, where $X_{j_i} \in \mathbb{R}$, $j = \left\{j_i\right\}_{i=1,\ldots,d}, j_i=0,\ldots,N_i-1$,  with
\begin{equation}
X_{j} = \frac 1 {\prod_{i=1}^{d} N_i} \sum\limits_{k=0}^{N-1} x_{k} \Wal(j,\frac k N),
\end{equation}
it follows that
\begin{equation}
\sum\limits_{j=0}^{N-1} |X_{j}|^2 = \frac 1 {\prod_{i=1}^{d} N_i} \sum\limits_{k=0}^{N-1} |x_{k}|^2.
\end{equation}
\end{lemma}

\begin{proof}
First, one observes that by definition and the fact that $N_i=2^{n_i}, i = 1,\ldots,d$ the following holds:
\begin{align}
|X_{j}|^2 & = \frac 1 {\prod_{i=1}^d N_i^2} \left( \sum\limits_{k=0}^{N-1} x_{k} \Wal(j,\frac {k} {N}) \right) 
\left( \sum\limits_{k=0}^{N-1} x_{k} \Wal(j,\frac {k} N) \right) \\
& = \frac 1 {\prod_{i=1}^d N_i^2} \sum\limits_{k=0}^{N-1} ~ \sum\limits_{l=0}^{N-1} x_{k} x_{l} \Wal(j_1,\frac {k_1 \oplus l_1} {N_1})\cdot \ldots \cdot \Wal(j_d, \frac {k_d \oplus l_d} {N_d}).
\end{align}
Next, recalling \eqref{EqWalAddition} and Lemma \ref{LemmaDyadicAdd} we directly get the desired property by
\begin{align}
\sum\limits_{j=0}^{N-1} |X_{j}|^2 &
= \frac 1 {\prod_{i=1}^d N_i^2} \sum\limits_{j=0}^{N-1} ~\sum\limits_{k=0}^{N-1}~ \sum\limits_{l=0}^{N-1}  x_{k} x_{l} \Wal(j_1,\frac {k_1 \oplus l_1} {N_1}) \cdot \ldots \cdot \Wal(j_d, \frac {k_d \oplus l_d} {N_d}) \\
& = \frac 1 {\prod_{i=1}^d N_i^2} \sum\limits_{k=0}^{N-1}~ \sum\limits_{l=0}^{N-1}  x_{k} x_{l} \sum\limits_{j_1=0}^{N_1-1} \Wal(j_1,\frac{k_1 \oplus l_1} {N_1})\cdot \ldots \cdot \sum\limits_{j_d=0}^{N_d-1} \Wal(j_d, \frac {k_d \oplus l_d} {N_d}) \\
& = \frac 1 {\prod_{i=1}^d N_i} \sum\limits_{k=0}^{N-1} |x_{k}|^2 .
\end{align}
\end{proof}
	With this we can achieve a useful relation between the Walsh polynomial, i.e.  $\Phi(z) = \sum_{j=A}^{B} \alpha_{j} \Wal(j,z)$ with $A, B \in \mathbb{Z}^d$ and $\alpha_{j_i} \in \mathbb{R}$ for all $j_i=A_i, \ldots, B_i$, $i=1,\ldots,d$, and its coefficients $\alpha$ similar to the trigonometric polynomial in \cite{linearity}.
	
	\begin{lemma}\label{Lemma2DphiSum}
	Let $A,B \in \mathbb{Z}^d$ such that $A_i \leq B_i, i =1, \ldots,d$ and consider the Walsh polynomial $\Phi(z) = \sum_{j=A}^{B} \alpha_{j} \Wal(j,z)$ for $z \in \mathbb{R}^d_+$. If $L = \left\{L_1, \ldots, L_d \right\}$ with $L_i = 2^{n_i}$, $n_i \in \mathbb{N}, i = 1,\ldots,d$ such that $2L_i \geq B_i - A_i +1$, then
	\begin{equation}
	\sum\limits_{j=0}^{2L-1} \frac 1 {\prod_{i=1}^d 2 L_i} \left| \Phi(\frac j {2 L})\right|^2 = \sum\limits_{j=A}^{B} |\alpha_{j}|^2 .
	\end{equation}
	\end{lemma}
	
	\begin{proof}
	For the proof Lemma \ref{PlancheralWalshLemma} is used. Therefore, let $x = \left\{x_{k}\right\}$, where $x_{k_i} \in \mathbb{R}$, $k=\left\{k_i\right\}_{i=1,\ldots,d}, k_i=0,\ldots,2L_i-1 $ and $x \in \mathbb{R}^{2L_1 \times \ldots \times 2L_d}$ with the discrete Walsh transformed $X = \left\{X_{j}\right\}  \in \mathbb{R}^{2L_1 \times \ldots \times 2L_d}$, where $X_{j_i} \in \mathbb{R}$, $j = \left\{j_i\right\}_{i=1,\ldots,d}, j_i=0,\ldots,2L_i-1$ and $X \in \mathbb{R}^{2L_1 \times \ldots \times 2L_d}$. Consider the sums $k_i+A_i$, $i = 1, \ldots,d$, there exists a number $\tilde{A}_i(k)$, such that $k_i + A_i = k_i \oplus \tilde{A}_i(k_i)$ for all $i =1, \ldots,d$. As before, we denote by $\tilde{A}(k)$ the multi-index containing all $\tilde{A}_i(k_i)$. Define the coefficients 
	\begin{equation}
	\tilde{\alpha}_{k_i} = \frac {\alpha_{k_i}}{\Wal( \frac {j_i} {2L}, \tilde{A}_i(k_i))}
	\end{equation}
	and the sequence $x$ as follows:
	\begin{align}
	x_{k+L} = 
	\begin{cases}
	\tilde{\alpha}_{k+A+L} & -L_i \leq k_i \leq -L_i + B_i -A_i \\
	0 & \text{otherwise} .
	\end{cases} 
	\end{align}
	Then one gets with the scaling property \eqref{Eq:scalingProperty} and the multiplicative identity \eqref{Eq:MultiplicativeID}
	\begin{align}
	X_{j} 
	& =  \frac 1 {\prod_{i=1}^d 2 L_i} \sum\limits_{k=0}^{2L -1}x_{k} \Wal(j,\frac k {2L})  = \frac 1 {\prod_{i=1}^d 2 L_i} \sum\limits_{k=-L}^{L-1} x_{k+L} \Wal(j,\frac {k+L} {2L}) \\
	& = \frac 1 { \prod_{i=1}^d 2 L_i} \sum\limits_{k=A}^{B} \tilde{\alpha}_{k} \Wal(\frac j {2L},k+A) 
	= \frac 1 { \prod_{i=1}^d 2 L_i} \sum\limits_{k=A}^{B} \tilde{\alpha}_k \Wal(\frac j {2L},k \oplus \tilde{A}(k)) \\
	& = \frac 1 { \prod_{i=1}^d 2 L_i} \sum\limits_{k=A}^{B} \tilde{\alpha}_{k} \Wal(\frac j {2L}, \tilde{A}(k))\Wal(\frac j {2L},k)  = \frac 1 {\prod_{i=1}^d 2 L_i} \sum\limits_{k=A}^{B}  \alpha_{k} \Wal(k,\frac j {2L})  = \frac 1 {\prod_{i=1}^d 2 L_i} \Phi(\frac j {2L}) .
	\end{align}
	With that one can conclude
	\begin{align}
	\sum\limits_{j=0}^{2L -1} \frac 1 { \prod_{i=1}^d 2 L_i} \left| \Phi(\frac j {2L}) \right|^2 
	&= \sum\limits_{j=0}^{2L -1} \prod_{i=1}^d 2 L_i |X_{j}|^2  = \sum\limits_{j=0}^{2L -1} |x_{j}|^2 =  \sum\limits_{j=A}^{B} |\tilde{\alpha}_{j}|^2 =  \sum\limits_{j=A}^{B} |\alpha_{j}|^2 .
	\end{align}
	\end{proof}
	
		\subsection{Changes of Wavelets}\label{Ch:WaveletChanges}
	As mentioned in chapter \ref{Ch:Walsh} the decimal and dyadic addition do not correspond directly to each other. In particular, the representation of the decimal addition with a number $h$ to a number $x$ depends on both parts of the sum. However, in Corollary \ref{Cor:scalingProp} we have seen that for $x \in [0,1)$ and $n \in \mathbb{N}$ the dyadic and decimal addition coincide. In the proof of the main theorem we want to transfer the time shifts of the wavelet to the Walsh function, i.e.
	\begin{align}
	\int\limits_{2^{-R}(n-p+1)}^{2^{-R}(n+p)} 2^{R/2} \phi (2^R x - n)\Wal( k, x)dx 
	& = 2^{-R/2} \int\limits_{-p+1}^{p} \phi(x)\Wal( k, 2^{-R}(x +n))dx \\
	& \neq 2^{-R/2} \int\limits_{-p+1}^{p} \phi(x)\Wal( k, 2^{-R}(x \oplus n))dx.
	\end{align}
	Therefore, to enable us to use Corollary \ref{Cor:scalingProp} and make the last equation an equality, the domain of the wavelets needs to be restricted to $[0,1]^d$. This is not a contradiction to the construction of the previous chapter, because the functions $\phi_{R,n}$ are indeed supported in $[0,1]^d$. However, the scaling function at level $0$ is not only supported there and that is the function that we are dealing with after the change of variables in the integral. To solve this problem we represent the scaling function as a sum of functions that are supported in $[0,1]$, i.e.
	\begin{equation}\label{phii}
	\phi(x) = \sum\limits_{i=-p+2}^p \phi_i (x-i+1) \text{ with } \phi_i(x) = \phi(x+i-1) \mathcal{X}_{[0,1]}(x)
	\end{equation}
	and
	\begin{equation}
	\phi_{R,n} = 2^{R/2}\sum\limits_{i=-p+2}^p \phi_i(2^Rx-i+1-n) .
	\end{equation}
	This can also be done accordingly for the reflected function $\phi^\#$. In the higher dimensional case we have
	\begin{equation}\label{EqTwoDimScFn}
	\phi(x) = \phi_1(x_1) \otimes \ldots \otimes \phi_d(x_d) = \sum\limits_{i=-p+2}^p \phi_{i_1} (x_1-i_1+1) \cdot \ldots \cdot \phi_{i_d} (x_d-i_2+1)
	\end{equation}
	and $\phi_{i_k}$ defined as above.
	This way the multiplicative identity holds also for the decimal time shift of the wavelets.

	\subsection{Useful lemma about wavelets}

	For the proof of the main theorem we have to bound the decay rate of the Wavelets under the Walsh transform. We analyse the Haar wavelets and the boundary-corrected Daubechies wavelet separately. 
	
	For the Haar wavelet we have that the scaling function $\phi = \mathcal{X}_{[0,1]}$. Hence, 
	\begin{align}
		\reallywidehat[W]{\phi}(\frac j {2^R}+m)
		&= \int_0^1 \Wal(\frac j {2^R}+m, x)dx 
		= \int_0^1 \Wal(j + 2^R m , \frac x {2^R})dx \\
		&= \int_0^{2^{-R}} \Wal(j+2^Rm, x) 2^R dx 
		= \begin{cases}
		1 \quad m = 0, j <2^R \\
		0 \quad \text{else.}
		\end{cases}
	\end{align}
	Therefore we have
	\begin{equation}\label{Eq:DecayHaar}
		\reallywidehat[W]{\phi}(\frac j {2^R}+m) \leq \frac {1}{m} \quad \text{ for } j=0,\ldots,2^R-1, m \geq 1.
	\end{equation}
		Next, we use the result from \cite{Vegard} regarding the decay rate of H\"older continuous functions. We have that the $\phi_i$ and $\phi_i^\#$ have the same smoothness properties as the original function $\phi$. Hence, we have for the Daubechies wavelet of order $2$ that it is H\"older continuous with coefficient $\alpha = 0.55$. The higher order wavelets are all H\"older continuous with coefficient $\alpha = 1$. The results in \cite{Vegard} hold only for inputs in $\mathbb{N}$. We will use the same proof strategy and notation for inputs in the form of $\frac k {2^R}$ for some $k,R \in \mathbb{N}$.
	
	\begin{lemma}\label{Lemma:Decayrate}
		Let $\phi$ be a Daubechies wavelet of order $p >1$. And let $m \geq 1$,$R \in \mathbb{N}$, $L=2^R$ and $j=0,\ldots,L-1$. Then we have that
		\begin{equation}
			|\reallywidehat[W]{\phi_i}(\frac j L + m)| \leq  \frac C {m^\alpha}
		\end{equation}
		and 
		\begin{equation}
			|\reallywidehat[W]{\phi_i^\#}(\frac j L + m)| \leq  \frac {C^\#} {m^\alpha}
		\end{equation}
		for some constants $C$ and $C^\#$.
	\end{lemma}
	
	\begin{proof}
		First, observe that 
	\begin{align}
		\int_{0}^1 \phi_i(x) \Wal(\frac j L + m , x) dx & = \int_0^1 \phi_i(x) \Wal(j +Lm, \frac x L)dx .
		\end{align}
		We want to use the same technique as in \cite{Vegard} to divide the integral into parts where the Walsh function takes the values $+1$ and $-1$ on intervals of the same length. Let $q \in \mathbb{N}$ such that $2^q \leq m < 2^{q+1}$, then we have that $2^{q+R} \leq j +Lm < 2^{q+R+1}$. Now, define the intervals $\Delta_k^{q} = \left[ 2^{-q}k, 2^{-q}(k+1) \right)$. It was shown that the function $\Wal(j + Lm, s)$ is constant on the interval $\Delta_{2l}^{q+R+1}$ and $\Delta_{2l+1}^{q+R+1}$ and takes the values $+1$ on one of them and $-1$ on the other. Hence, $\Wal(j + Lm, \frac x L)$ takes these values on the intervals $\Delta_{2k}^{q+1}$ and $\Delta_{2k+1}^{q+1}$. Due to the H\"older continuity we have that there exists a constant $C$ such that $\phi_i(x) \leq \phi_i(s) + C|x-s|^\alpha$ for all $s \in [0,1]$. With this one gets 
		\begin{align}
			\sup_{x \in \Delta_{k}^{q}} \phi_i(x) & \leq \phi_i(2^{-q}k + 2^{-q-1}) + C2^{-(q+1)\alpha} \\
			\sup_{x \in \Delta_{k}^{q}} -\phi_i(x) & \leq - \phi_i(2^{-q}k + 2^{-q-1}) + C2^{-(q+1)\alpha}.
		\end{align}
		Hence,
		\begin{align}
			| \int_{\Delta_{k}^q} \phi_i(x) &\Wal(j+Lm,\frac x L)dx | 
			\\& \leq 2^{-q}|(\phi_i(2^{-q}k + 2^{-p-1})+ C2^{-(p+1)\alpha}) \ + (-\phi_i(2^{-q}k + 2^{-p-1})+ C2^{-(p+1)\alpha})| \\
			&\leq 2^{-q}C 2^{-q\alpha}.
		\end{align}
		Thus, we get for the complete integral on $[0,1]$
		\begin{align}
		\int_{0}^1 \phi_i(x) \Wal(\frac j L + m , x) dx & = \sum_{k = 0}^{2^{q}+1} \int_{\Delta_k^{q}} \phi(x) \Wal(j + Lm, \frac x L)dx \\
		& \leq \sum_{k = 0}^{2^{q}+1} 2^{-q} |2^{-q}|^\alpha  \leq C 2^{-q\alpha} \leq \frac {2C} {m^\alpha}.
	\end{align}
	The proof holds analogously for $\phi_i^\#$. 
	\end{proof}

	\subsection{Proof of the main theorem}
	With the tools established above, we can now prove the main result. To make the exposition easier to read we first prove the theorem in one dimension and then make the generalisation to several dimensions in a separate proof. Given the setup with the multi-indices framework, this can be done reasonably smoothly.

	\begin{proof}[Proof of Theorem \ref{Theo41CH2} in one dimension]
		The aim of this proof is to find for every $\theta \in (1,\infty)$ an integer $S_\theta \in \mathbb{N}$, such that for all $M \geq S_\theta N$ the subspace angle is bounded, i.e. $\mu(\mathcal{R}_N,\mathcal{S}_M) \leq \theta$. Let $R \in \mathbb{N}$ be the number of reconstructed levels, i.e. $N = 2^R$. We start with a suitable representation of $\cos(\omega(\mathcal{R}_N,\mathcal{S}_M)) $. There exists $\varphi \in \mathcal{R}_N$ with $||\varphi||=1$ such that $\inf_{f \in \mathcal{R}_N, ||f||=1} || P_{\mathcal{S}_M} f|| = || P_{\mathcal{S}_M} \varphi||$, because the closed unit ball in $\mathcal{R}_N$ is compact and $P_{\mathcal{S}_M}$ is continuous. By \eqref{Eq:repvarphi1d} we can represent $\varphi$ as
		\begin{equation}\label{Eq:proof1varphi}
		\varphi = \sum\limits_{l=0}^{2^R-p-1} \alpha_l \phi_{R,n} + \sum\limits_{l=2^R -p}^{2^R-1} \beta_l \phi_{R,n}^\# \text{ with } \sum\limits_{l=0}^{2^R-p-1} |\alpha_n|^2 + \sum\limits_{l = 2^R -p}^{2^R-1} |\beta_n|^2 = 1 
		\end{equation}
		and
		\begin{align}\label{EqCNM}
		\cos(\omega(\mathcal{R}_N,\mathcal{S}_M)) & = \inf_{f \in \mathcal{R}_N, ||f||=1} || P_{\mathcal{S}_M} f|| = || P_{\mathcal{S}_M} \varphi || \\ & = || \varphi - P_{\mathcal{S}_M}^\perp \varphi ||  \geq ||\varphi|| - || P_{\mathcal{S}_M}^\perp \varphi || = 1 - || P_{\mathcal{S}_M}^\perp \varphi || .
		\end{align}
		The first equation \eqref{Eq:proof1varphi} allows us to deal only with the scaling function instead of both the wavelets at different scales and the scaling function. The second one \eqref{EqCNM} enables us to bound $P_{\mathcal{S}_M}^\perp \varphi$ from above in lieu of $P_{\mathcal{S}_M} \varphi$ from below.
		
		Instead of dealing with all different shifts of the scaling function, we aim for a closed form that only depends on the functions $\phi$ and $\phi^\#$. An essential part in the construction of this is the use of the scaling property in Corollary \ref{Cor:scalingProp}. Therefore, it is necessary, that $m \in \mathbb{N}$ and $x \in [0,1)$. For this sake, the functions $\phi_i$ were defined in \eqref{phii} and we define 
		\begin{equation}\label{Eq:p_R}
			p_R: \mathbb{Z} \rightarrow \mathbb{N}
		\end{equation}
	 with $z  \mapsto p_R(z)$ and $p_R(z)$ being the smallest integer such that $p_R(z) 2^R + z >0$. This yields 
		\begin{align}\label{Eq:WalshShift1}
		\langle \phi_{R,n}  , \Wal(k, \cdot) \rangle 
		&= \sum\limits_{i=-p+2}^p \langle \phi_{i,R,n} , \Wal( k, \cdot) \rangle  \\
		& = 2^{R/2} \sum\limits_{i=-p+2}^p ~ \int\limits_{2^{-R}(n+i-1)}^{2^{-R}(n+i)} \phi_i (2^R x - n -i +1) \Wal( k, x)dx \\
		& = 2^{-R/2} \sum\limits_{i=-p+2}^p \int\limits_0^1 \phi_i (x)\Wal( k, 2^{-R}(x+n+i-1))dx \\
		& =  2^{-R/2} \sum\limits_{i=-p+2}^p \int\limits_0^1 \phi_i (x)\Wal( k, 2^{-R}(x+n+i-1 + 2^R p_R(i-1)))dx,
		\end{align}
		where we used in the last line the fact that the Walsh functions are $1$-periodic, if the other input data is an integer.
		Then, we have that $x \in [0,1]$ and $n+i-1 + 2^R p_R(n+i-1) \in \mathbb{N}$. Hence, Corollary \ref{Cor:scalingProp} can be used in the third line to get
		\begin{align}\label{Eq:WalshShift2}
		&\langle \phi_{R,n}  , \Wal(k, \cdot) \rangle \\
			& =  2^{-R/2} \sum\limits_{i=-p+2}^p \int\limits_0^1 \phi_i (x)\Wal( k, 2^{-R}(x +(n+i-1 + 2^R p_R(i-1))))dx \\
		& = 2^{-R/2} \sum\limits_{i=-p+2}^p \Wal(k, 2^{-R}(n+i-1 + 2^R p_R(i-1))) \int\limits_0^1 \phi_i(x)\Wal( k, 2^{-R}x)dx \\
		& = 2^{-R/2} \sum\limits_{i=-p+2}^p \Wal(n+i-1 + 2^R p_R(i-1), \frac{ k}{2^{R}})\reallywidehat[W]{\phi_i}(\frac{ k}{2^R}).
		\end{align}
		With 
		\begin{align}\label{Eq:WalshPolyProof}
		\Phi_i(z) = \sum\limits_{n=0}^{2^R-p-1} \alpha_n \Wal(n+i-1 + 2^R p_R(i-1),z)
		\end{align}
		it results in
		\begin{align}\label{EqTrick}
		\sum\limits_{n=0}^{2^R-p-1}  \alpha_n  \langle \phi_{i,R,n} , \Wal( k, \cdot) \rangle 
		& =  2^{-R/2} \sum\limits_{n=0}^{2^R-p-1} \alpha_n \Wal(n+i-1 +  2^R p_R(i-1),\frac{ k}{2^R})  \reallywidehat[W]{\phi_i}(\frac{ k}{2^R}) \\ &  =  2^{-R/2} \reallywidehat[W]{\phi_i}(\frac{ k}{2^R}) \Phi_i(\frac{ k}{2^R}).
		\end{align}
		Analogously this can be done for the reflected function $\phi^\#$. Thus, by using
		\begin{equation}
		\Phi_i^\# (z) = \sum\limits_{n = 2^R -p}^{2^R-1} \beta_n \Wal(n + i -1+ 2^R p_R(i -1),z)
		\end{equation} 
		similar to the above, we get
		\begin{align}\label{Eq:WalshHash}
		\sum\limits_{n = 2^R -p}^{2^R-1}  \beta_n  \langle \phi_{i,R,n}^\# , \Wal( k, x) \rangle  
		&=  2^{-R/2} \sum\limits_{n=2^R - p}^{2^R-1} \beta_n \Wal(n+i-1 +  2^R p_R(i-1),\frac{ k}{2^R}) \reallywidehat[W]{\phi_i^\#}(\frac{ k}{2^R}) \\
		& =  2^{-R/2} \reallywidehat[W]{\phi_i^\#}(\frac{ k}{2^R}) \Phi_i^\#(\frac{ k}{2^R}).
		\end{align}
		This representation is very useful. Indeed, one only has to pay attention to the decay rate of the Walsh transform of the pieces of the mother scaling function, its reflection and the Walsh polynomial. Moreover, the pieces fulfil $f(t) = 0$ for $t <0$, such that the Walsh transform with one kernel can be used for the analysis. Consider
		\begin{align}\label{Eq:PlugInCutOf1D}
		||P_{\mathcal{S}_M}^\perp \varphi|| 
		& = || P_{\mathcal{S}_M}^\perp (\sum\limits_{n=0}^{2^R-p-1} \alpha_n \phi_{R,n} + \sum\limits_{n=2^R-p}^{2^R-1} \beta_n \phi_{R,n}^\#)|| \\
		& = || P_{\mathcal{S}_M}^\perp (\sum\limits_{n=0}^{2^R-p-1} \alpha_n \sum\limits_{i=-p+2}^p \phi_{i,R,n} + \sum\limits_{n=2^R-p}^{2^R-1} \beta_n \sum\limits_{i=-p+2}^p \phi_{i,R,n}^\# )|| .
		\end{align}
		Next, one uses the linearity of the orthogonal projection to change the order of the summands, such that the sum over the scaling function pieces can be dealt with in the end.
		First, we take out the sum over the parts of the wavelet $\phi_{i,R,n}$ to handle every cut-out of the wavelet separately. In particular, by \eqref{Eq:PlugInCutOf1D}
		\begin{align}\label{Eq:OrderCutOut}
		||P_{\mathcal{S}_M}^\perp \varphi|| & = || \sum\limits_{i=-p+2}^p P_{\mathcal{S}_M}^\perp (\sum\limits_{n=0}^{2^R-p-1} \alpha_n \phi_{i,R,n} + \sum\limits_{n=2^R-p}^{2^R-1} \beta_n \phi_{i,R,n}^\# ) || \\
		 & \leq \sum\limits_{i=-p+2}^p || P_{\mathcal{S}_M}^\perp
		(\sum\limits_{n=0}^{2^R-p-1} \alpha_n \phi_{i,R,n} + \sum\limits_{n=2^R-p}^{2^R-1} \beta_n \phi_{i,R,n}^\#) || \\
		& =  \sum\limits_{i=-p+2}^p \sqrt{\sum\limits_{k > M} \left|\sum\limits_{n=0}^{2^R-p-1} \alpha_n \langle \phi_{i,R,n}, \Wal(k, \cdot)\rangle + \sum\limits_{n=2^R-p}^{2^R-1} \beta_n \langle \phi_{i,R,n}^\# , \Wal(k,\cdot) \rangle \right|^2}.
		\end{align}
	Second, it follows, by using \eqref{EqTrick}, \eqref{Eq:WalshHash} and the Cauchy-Schwarz inequality, that
		\begin{align}\label{Eq:AfterCauchy1d}
		||P_{\mathcal{S}_M}^\perp \varphi|| 
		& \leq \sum\limits_{i=-p+2}^p \sqrt{\sum\limits_{k > M} 2^{-R} \left| \reallywidehat[W]{\phi_i}(\frac k {2^R}) \Phi_i (\frac k {2^R}) + \reallywidehat[W]{\phi^\#_i}(\frac k {2^R}) \Phi_i^\# (\frac k {2^R}) \right|^2} \\
		& \leq  \sum\limits_{i=-p+2}^p \left( \sum\limits_{k \geq M} 2^{-R} \left| \reallywidehat[W]{\phi_i}(\frac k {2^R}) \Phi_i (\frac k {2^R}) \right|^2 + \sum\limits_{k \geq M} 2^{-R} \left|\reallywidehat[W]{\phi^\#_i}(\frac k {2^R}) \Phi_i^\# (\frac k {2^R}) \right|^2 \right. \\
		& \left. + 2 \left( \sum\limits_{k \geq M} 2^{-R} \left| \reallywidehat[W]{\phi_i}(\frac k {2^R}) \Phi_i (\frac k {2^R}) \right|^2 \right)^{1/2} \left( \sum\limits_{k \geq M} 2^{-R} \left| \reallywidehat[W]{\phi^\#_i}(\frac k {2^R}) \Phi_i^\# (\frac k {2^R}) \right|^2 \right)^{1/2} \right)^{1/2}.
		\end{align}
		We will only deal with the first summand and the other follow analogously. In the following step the $1$- periodicity of the Walsh function is used. Let $S \in \mathbb{R}_+$ such that $M = S 2^R$, then by replacing $k\geq M$ with $k = mL + j$ with $L =  2^R$, $m\geq S$ and $j = 0,\ldots,L-1$ we have
		\begin{align}\label{Eq:k=mL+j}
		\sum\limits_{k \geq M} 2^{-R} \left| \reallywidehat[W]{\phi_i}(\frac k {2^R}) \Phi_i (\frac k {2^R}) \right|^2 
		& \leq   \sum\limits_{j=0}^{L-1} \frac{1}{L } \left| \Phi_i(\frac{ j}{L}) \right|^2 \sum\limits_{m \geq S} \left|\reallywidehat[W]{\phi_i}\left(\frac j L + m\right)\right|^2 .
		\end{align}
		The last sum can be estimated via Lemma \ref{Lemma:Decayrate} and \eqref{Eq:DecayHaar} by
		\begin{align}\label{Eq:EstimateSum2}
		\sum\limits_{m \geq S} \left|\reallywidehat[W]{\phi_i}\left(\frac j L + m\right)\right|^2 
		& \leq \sum\limits_{m \geq S} \frac {A^2}{ m^{2\alpha}}  \leq \frac{ A^2}{ S^{2\alpha -1}}
		\end{align}
		with $\alpha =1$ for the Haar wavelet and $\alpha$ being the H\"older coefficient for the other Daubechies wavelets of order $p >1$. For the first sum $\sum\limits_{j=0}^{L-1} \frac{1}{L } \left| \Phi_i(\frac{ j}{L}) \right|^2$ we have with $\alpha_n = 0$ for $n = -p+1, \ldots, -1$ that
		\begin{align}\label{Eq:PhiRep}
		\Phi_i(z) & = \sum\limits_{n=0}^{2^R-p} \alpha_n \Wal(n + i -1 +  2^R p_R(i-1),z) \\ 
		& = \sum\limits_{n=-p+1}^{2^R-p} \alpha_l \Wal(n + i -1 +  2^R p_R(i-1),z) \\
		& = \sum\limits_{n = -p+1 + i -1 +  2^R p_R(i-1)}^{2^R - p + i - 1 + 2^R p_R(i-1)} \alpha_{n -i +1 -2^R p_R(i-1)} \Wal(n,z) 
		\end{align}
		and $L = 2^R$, such that we can use Lemma \ref{Lemma2DphiSum} as $2^R -p + i -1 + 2^R p_R(i-1) - (-p +1 + i -1 +  2^R p_R(i-1)) + 1 = 2^R -p -( -p +1) +1 = 2^R \geq L$ and obtain
		\begin{equation}\label{Eq:BoundPhi}
		\sum\limits_{j=0}^{L-1} \frac 1 L \left| \Phi_i (\frac j L) \right|^2 = \sum\limits_{l =  -p +1 + i -1 +  2^R p_R(i-1)}^{ 2^R -p + i - 1 +  2^R p_R(i-1)} |\alpha_{l -i +1 -2^R p_R(i-1)}|^2 = \sum\limits_{n =  -p +1}^{ 2^R -p } |\alpha_{n }|^2 \leq 1.
		\end{equation} 
		Altogether this gives with Lemma \ref{Lemma:Decayrate}
		\begin{equation}\label{Eq:Boundphi}
		\sum\limits_{k \geq M} 2^{-R} \left| \reallywidehat[W]{\phi_i}(\frac k {2^R}) \Phi_i (\frac k {2^R}) \right|^2  \leq  \frac{A^2}{S^{2\alpha -1}},
		\end{equation}
		and similarly
		\begin{equation}\label{Eq:Boundphi1}
		\sum\limits_{k \geq M} 2^{-R} \left| \reallywidehat[W]{\phi_i^\#}(\frac k {2^R}) \Phi_i^\# (\frac k {2^R}) \right|^2  \leq \frac{{A^\#}^2}{S^{2\alpha -1}}.
		\end{equation}
		Using \eqref{Eq:BoundPhi}, \eqref{Eq:Boundphi} and \eqref{Eq:Boundphi1} yields the following estimation
		\begin{align}\label{Eq:Finalbound}
		||P_{\mathcal{S}_M}^\perp \varphi|| 
		& \leq \sum\limits_{i=-p+2}^{p} \left( \frac {A^2} {S^{2\alpha -1}} + \frac{{A^\#}^2}{S^{2\alpha -1}} + 2 \frac {A A^\# }{S^{2\alpha -1}} \right)^{1/2} \\
		& \leq (2p-2) \left( \frac{4 \max\left\{A^2, {A^\#}^2\right\}}{S^{2\alpha -1}} \right) ^{1/2} = (2p-2) \left( \frac {C^2} {S^{2\alpha -1}} \right)^{1/2}.
		\end{align}
		Thus, $||P_{\mathcal{S}_M}^\perp \varphi|| \leq \gamma$ whenever
		\begin{equation}\label{EqS}
		S \geq   \left(\frac{C (2p-2)}{\gamma }\right)^{2/2\alpha-1},
		\end{equation}
		where $C = 4 \max\left\{A^2, {A^\#}^2\right\}$.
		It follows from \eqref{EqCNM} that 
		$
		\cos(\omega(\mathcal{R}_N,\mathcal{S}_M)) \geq 1 - \gamma \geq  \frac 1 {\theta}
		$ 
		,i.e. $\mu(\mathcal{R}_N,\mathcal{S}_M) \leq \theta$, whenever the constant $S$, which is dependent on $\theta$ and therefore denoted by $S_\theta$, fulfils \eqref{EqS} with $\gamma =  1 - \frac 1 \theta$, i.e 
		\begin{equation}\label{Eq:Stheta}
		S_\theta \geq  \left(\frac{C (2p-2) \theta}{\theta -1 }\right)^{2/2\alpha -1}
		\end{equation} and $M = S_\theta L = S_\theta N$.
	\end{proof}

\begin{proof}[Proof of Theorem \ref{Theo41CH2} in $d$-dimensions]
	In higher dimensions we represent $\varphi \in \mathcal{R}_N$ in terms of the sum over $2^d$ different tensor products. Then, we need to investigate the inner products of these summands with the Walsh functions as in the one dimensional case. At this point the results from the one dimensional case come into play. Next, one investigates the parts of the set $I_M^\perp$, where $I_M = \left\{ l = \left\{l_1,\ldots,l_d\right\}, l_k = 0,\ldots,M_k-1 \right\}$ and $M = \left\{M_1,\ldots,M_d\right\}\in \mathbb{N}^d$ and $I_M^\perp = \mathbb N^d \setminus I_M$, which correspond to the largest estimates of the inner products of the summands of the wavelet and the Walsh function. Finally, these can be bounded with estimates from the one dimensional case and additional care for the finite sums.
	
	Now, we present the described steps in more detail.	Let $\varphi \in \mathcal{R}_N$ with $|| \varphi || =1$. Then we can represent $\varphi$ as in \eqref{Eq:R_Ndd} in the following sum
	\begin{equation}
	\varphi = \sum\limits_{s \in \left\{ 0,1 \right\}^d} \sum\limits_{n \in K_s} \alpha_n \bigotimes \phi_{R,n}^s .
	\end{equation}
	In the one dimensional case we derived the representation of 
	$\sum_{n=0}^{2^R-p-1}  \alpha_n  \langle \phi_{i,R,n} , \Wal( k, \cdot) \rangle $ in terms of the Walsh transform of the wavelet and the Walsh polynomial, i.e. $2^{-R/2} \reallywidehat[W]{\phi_i}(\frac{ k}{2^R}) \Phi_i(\frac{ k}{2^R})$. This equality from \eqref{EqTrick} is used to represent the inner product in higher dimensions. For this we need to define $p_R$ from \eqref{Eq:p_R} for higher dimensions. In particular, let $p_R : \mathbb{Z}^d \rightarrow \mathbb{N}^d$ with $\left\{z_i\right\}_{i=1,\ldots,d}= z \mapsto p_R(z) = \left\{p_R(z)_i\right\}_{i=1,\ldots,d}$ and $p_R(z)_i$ being the smallest integer such that $p_R(z)_i 2^R - z_i >0$ for all $i=1, \ldots,d$. Further, let $l =(l_1,\ldots,l_d) \in \mathbb{Z}^d$. This yields 
	\begin{align}
	\langle \phi_{i,R,n}^s, \Wal(l,\cdot) \rangle 
	& = \prod_{k=1}^d \langle \phi_{i,R,n_k}^{s_k}, \Wal(l_k, \cdot) \rangle. 
	\end{align}
	Here, the problem is reduced to the one dimensional case and we can apply \eqref{Eq:WalshShift1} and \eqref{Eq:WalshShift2} to get
	\begin{align}
	\langle \phi_{i,R,n}^s, \Wal(l,\cdot) \rangle & =\prod_{k=1}^d  2^{-dR/2} \Wal(n_k + i_k -1 + 2^R p_R(i_k -1),\frac {l_k}{2^R}) \reallywidehat[W]{\phi_{i_k}^{s_k}}(\frac{l_k}{2^R}) \\
	& =  \Wal(n + i -1 + 2^R p_R(i -1),\frac {l}{2^R}) \reallywidehat[W]{\phi_i^s}(\frac l {2^R}) .
	\end{align}
	This is defined now as in \eqref{Eq:WalshPolyProof}
	\begin{align}
	\Phi_i^s (z) & 
	= \sum\limits_{n \in K_s} \alpha_{n} \Wal(n+ i -1 + 2^R p_R(i -1),\frac {z}{2^R}).
	\end{align}
	Note that the different definitions from the one dimensional case for $\Phi$ and $\Phi^\#$ are combined in the notation with the $K_s$. We get with this the presentation of the inner products as desired: 
	\begin{align}
	\sum\limits_{n \in K_s} \alpha_{n} \langle \phi_{i,R,n}^s, \Wal(l,\cdot) \rangle 
	&  =  2^{-dR/2} \Phi_i^s (\frac l {2^R}) \widehat{\phi_i^s}(\frac l {2^R}). 
	\end{align}
	For the representation of indices that correspond to the sampling functions, let $I_M = \{ l = \{l_1,\ldots,l_d \}, l_k = 0,\ldots,M_k-1 \}$, where $M = \{M_1,\ldots,M_k \}\in \mathbb{N}^d$ is the number of samples. Then $l \notin I_M$ corresponds to $l > m$ in the one dimensional case. We now want to analyse the orthogonal projection on the orthogonal complement of the sampling space
	\begin{align}
	||P_{\mathcal{S}_M}^\perp \varphi|| 
	& = || P_{\mathcal{S}_M}^\perp (\sum\limits_{s \in \left\{ 0,1 \right\}^d} \sum\limits_{n \in K_s} \alpha_{n} \phi_{R,n}^s) || 
	 = || P_{\mathcal{S}_M}^\perp (\sum\limits_{s \in \left\{ 0,1 \right\}^d} \sum\limits_{n \in K_s} \alpha_{n} \sum\limits_{i = -p+2}^p \phi_{i,R,n}^s) || .
	 \end{align}
	 This way we ensured the use of Corollary \ref{Cor:scalingProp}. Next we change the order again to deal with the different cut-out functions separately. This was seen already in \eqref{Eq:OrderCutOut}. We get
	 \begin{align}
	||P_{\mathcal{S}_M}^\perp \varphi|| & = || \sum\limits_{i = -p+2}^p P_{\mathcal{S}_M}^\perp (\sum\limits_{s \in \left\{ 0,1 \right\}^d} \sum\limits_{n \in K_s} \alpha_{n} \sum\limits_{i = -p+2}^p \phi_{i,R,n}^s ) || 
	 \leq \sum\limits_{i = -p+2}^p || P_{\mathcal{S}_M}^\perp (\sum\limits_{s \in \left\{ 0,1 \right\}^d} \sum\limits_{n \in K_s} \alpha_{n} \sum\limits_{i = -p+2}^p \phi_{i,R,n}^s ) ||.
	 \end{align}
	 With the Cauchy Schwarz inequality and a careful reordering we get as in \eqref{Eq:AfterCauchy1d}
	 \begin{align}
		||P_{\mathcal{S}_M}^\perp \varphi|| & =  \sum\limits_{i= -p+2}^p \sqrt{\sum\limits_{l \notin I_\mathcal{M}} \left| \sum\limits_{s \in \left\{ 0,1 \right\}^d} \sum\limits_{n \in K_s} \alpha_{n} \langle \phi_{i,R,n}^s,\Wal(l,\cdot) \rangle \right|^2} 
	 =  \sum\limits_{i = -p+2}^p \sqrt{\sum\limits_{l \notin I_\mathcal{M}}2^{-dR} \left| \sum\limits_{s \in \left\{ 0,1 \right\}^d} \Phi_i^s (\frac l {2^R}) \reallywidehat[W]{\phi_i^s}(\frac l {2^R})\right|^2} \\
	 & \leq \sum\limits_{i = -p+2}^p \sum\limits_{s \in \left\{0,1 \right\}^d} \left( \sum\limits_{l \notin I_\mathcal{M}} 2^{-dR} | \Phi_i^s (\frac l {2^R}) \reallywidehat[W]{\phi_i^s}(\frac l {2^R}) |^2 \right)^{1/2}  .
	\end{align}
	
	Now, let $S \in \mathbb{N}$ be given, such that the number of samples $M = \left\{M_1,\ldots,M_k\right\} \in \mathbb{N}^d$ is $M_k= S2^R$. Then if $l = \left\{l_1,\ldots,l_d\right\} \notin I_\mathcal M$ at least one $l_k > M_k$. The sum is the largest if only one $l_k$ fulfils this estimate. Hence, without loss of generality let $l_1 > M_1$ and $l_k \leq M_k$ for $k = 2, \ldots, d$. Now let $l_k = j_k + u_k 2^R$ and $\alpha=0.55$ or $\alpha =1$ depending on the chosen wavelet. Then we get similar to \eqref{Eq:EstimateSum2}:
	\begin{align}
	\sum\limits_{l_1 > M_1} \sum\limits_{l_2 \leq M_2} \ldots \sum\limits_{l_d \leq M_d}  2^{-dJ} \left|  \Phi_i^s (\frac l {2^R}) \reallywidehat[W]{\phi_i^s} (\frac l {2^R}) \right|^2
	& = \sum\limits_{j = 0}^{2^R-1} \frac 1 {2^{dR}} \left| \Phi_i^s (\frac j {2^R}) \right|^2 \sum\limits_{u_1 > S} \sum\limits_{u_2 \leq S}\ldots\sum\limits_{u_d \leq S} \left| \phi_i^s(\frac j {2^R} + u) \right|^2 \\
	& \leq \sum\limits_{j = 0}^{2^R-1} \frac 1 {2^{dR}} \left| \Phi_i^s (\frac j {2^R}) \right|^2 \sum\limits_{u_1 > S} \sum\limits_{u_2 \leq S}\ldots\sum\limits_{u_d \leq S} \frac{A_1^2}{(1+u_1)^{2\alpha}}\ldots \frac{A_d^2}{(1+u_d)^{2\alpha}} \\
	& \leq  \frac {C^{d-1}} {S^{2\alpha -1}}  \sum\limits_{j = 0}^{2^R-1} \frac 1 {2^{dR}} \left| \Phi_i^s (\frac j {2^R}) \right|^2 .
	\end{align}
	The last sum can be estimated with the help of Lemma \ref{Lemma2DphiSum}. In \eqref{Eq:PhiRep} and \eqref{Eq:BoundPhi} this was derived in the one dimensional case that can be directly used here, such that
	\begin{align}
	\sum\limits_{j = 0}^{2^R-1} \frac 1 {2^{dR}} \left| \Phi_i^s (\frac j {2^R}) \right|^2 =  \sum\limits_{n \in K_s}  |\alpha_{n}|^2.
	\end{align} 
	From the fact that the $\phi_{R,n}^s$ form an orthonormal bases and $||\varphi|| = 1$ we have
	\begin{equation}
	\sum\limits_{n \in K_s}|\alpha_n|^2 \leq 1 .
	\end{equation}
	This together with the fact that $|\left\{0,1\right\}^d| = 2^d$ gives
	\begin{equation}
\sum\limits_{s \in \left\{0,1 \right\}^d} \left( \sum\limits_{l_1 > M_1} \sum\limits_{l_2 \leq M_2} \ldots \sum\limits_{l_d \leq M_d} 2^{-dR} \left|  \Phi_i (\frac l {2^R}) \widehat{\phi_i} (\frac l {2^R}) \right|^2 \right)^{1/2}
	\leq 2^d \left( \frac {C^{d-1} } {S^{2\alpha -1}} \right)^{1/2}.
	\end{equation}
		By replacing $C^{d-1}$ with $C$ we have
	\begin{equation}
	|| P_{\mathcal{S}_M}^\perp \varphi || \leq \sum\limits_{i = -p+2}^p 2^d \left( \frac {C} {  S} \right)^{1/2} = (2p-2) 2^d \left( \frac {C} {  S^{2\alpha -1}} \right)^{1/2}.
	\end{equation}
	Thus, $||P_{\mathcal{S}_M}^\perp \varphi || \leq \gamma$ whenever
	\begin{equation}\label{EqS2D}
	S \geq C \left( \frac {(2p-2)2^d}{\gamma} \right)^{2/2\alpha -1}  .
	\end{equation}
	It follows from \eqref{EqCNM} that 
	\begin{equation}
		\cos(\omega(\mathcal{R}_N,\mathcal{S}_M)) \geq 1 - \gamma \geq \frac{1}{\theta},
	\end{equation}
	i.e. $\mu(\mathcal{R}_N,\mathcal{S}_M) \leq \theta$, whenever $S_\theta$ fulfils \eqref{EqS2D} with $\gamma = 1 - \frac 1 \theta$, i.e.
	\begin{equation}
	S_\theta \geq  C \left( \frac {(2p-2) 2^d \theta }{\theta -1} \right)^{2/2\alpha -1}   
	\end{equation}
	and $|M|= M_1 \cdot \ldots \cdot M_d = S_\theta^d 2^{dR} = S_\theta^d N$.
\end{proof}

\section{Numerical Experiments}

In this chapter we underline the theoretical results with numerical experiments. We first calculate the stable sampling rate for different stabilities $\theta$ and Daubechies wavelets. Then, we see that the reconstruction with generalized sampling leads to much better results then the direct inversion with the Walsh transform. Moreover, we point out that it is important to consider the stable sampling rate, as otherwise the reconstruction gets very unstable with meaningless results. 

First, we see in Figure \ref{fig:StableSamplingRate} the stable sampling rate for the different wavelets and stabilities in the one-dimensional case. One can see that it is indeed linear with jumps according to the levels of the wavelets. Moreover, it is easy to detect that the constant $S_\theta$ is considerably low such that the number of samples needed is only marginally larger than the number of coefficients that we reconstruct. It is not surprising that the stable sampling rate gets larger for smaller $\theta$. In the theory of the reconstruction from Fourier measurements we have a direct relation between the smoothness of the wavelets and the size of the stable sampling rate. Similar relations are not known for the Walsh wavelet case.

\begin{figure}[ht]
	\centering
	\subfloat[$\Theta(N;5)$ for DB2 with $\frac {N_{max}}{M_{max}} = 1.249$]{
		\includegraphics[width=0.45\textwidth]{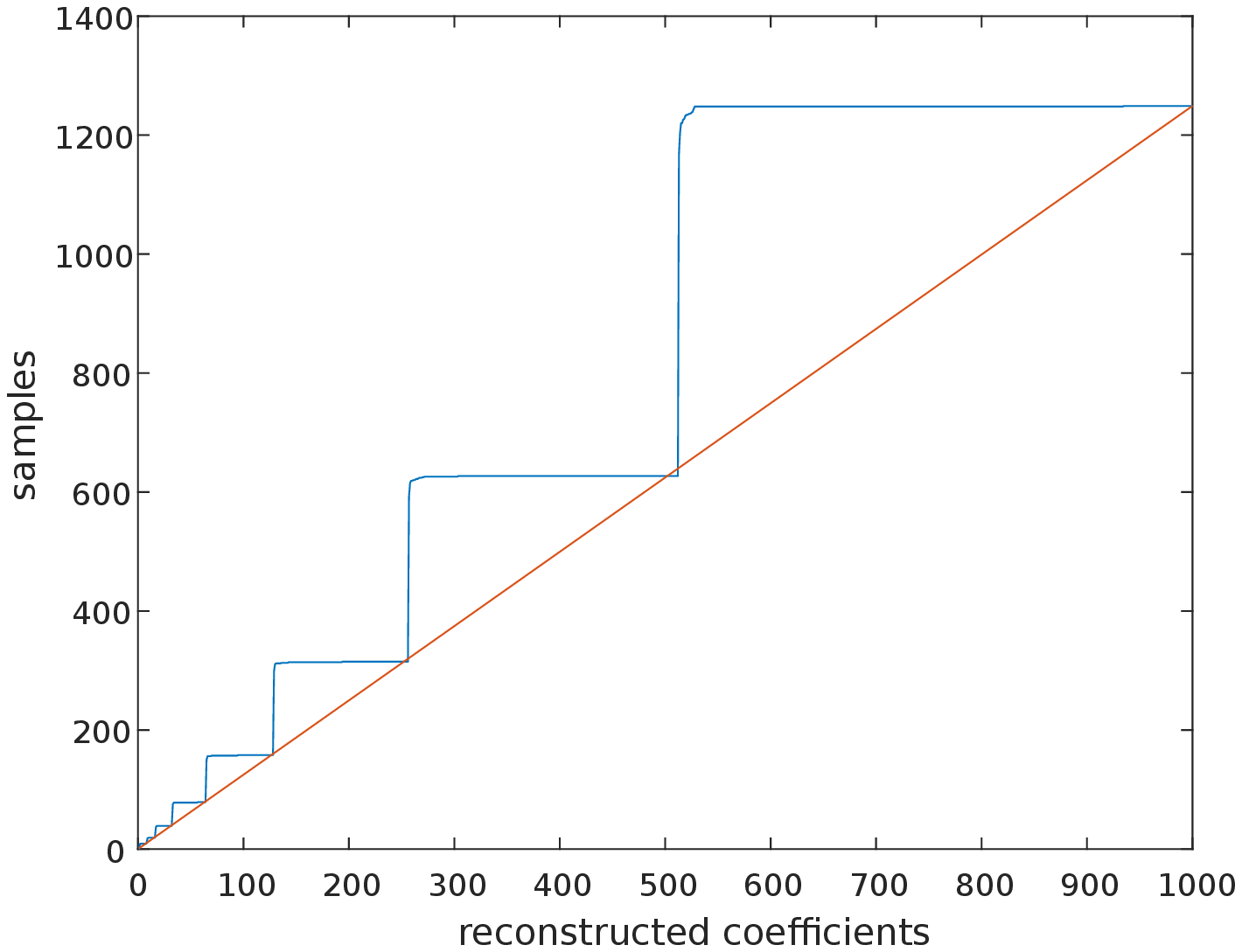}
		\label{fig:DB2Theta5}}
	\quad
	\subfloat[$\Theta(N;2)$ for DB2 with $\frac {N_{max}}{M_{max}} = 1.49$]{
		\includegraphics[width=0.45\textwidth]{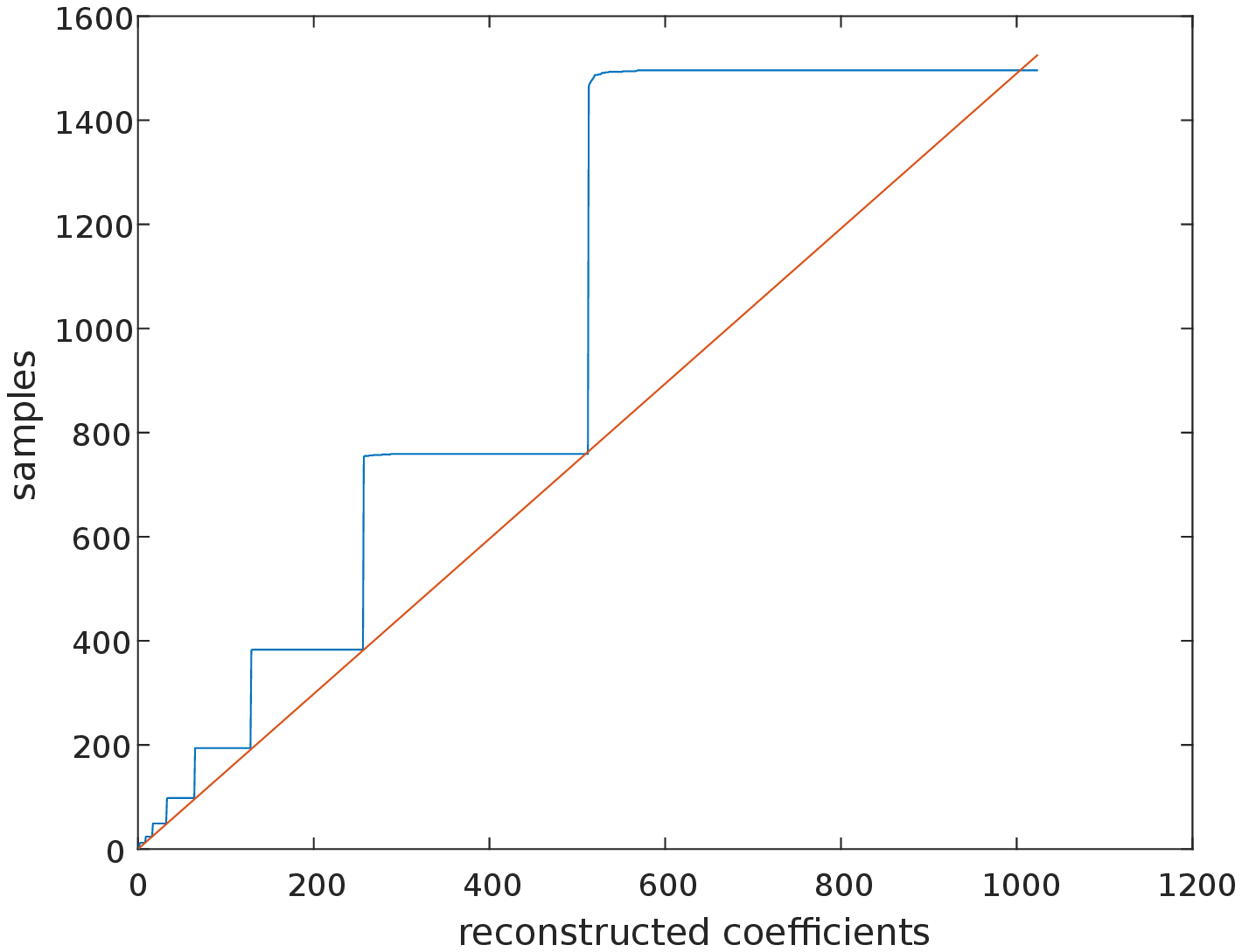}
		\label{fig:DB2Theta2}}
	
	\subfloat[$\Theta(N;5)$ for DB8 with $\frac {N_{max}}{M_{max}} = 1.257$]{
		\includegraphics[width=0.45\textwidth]{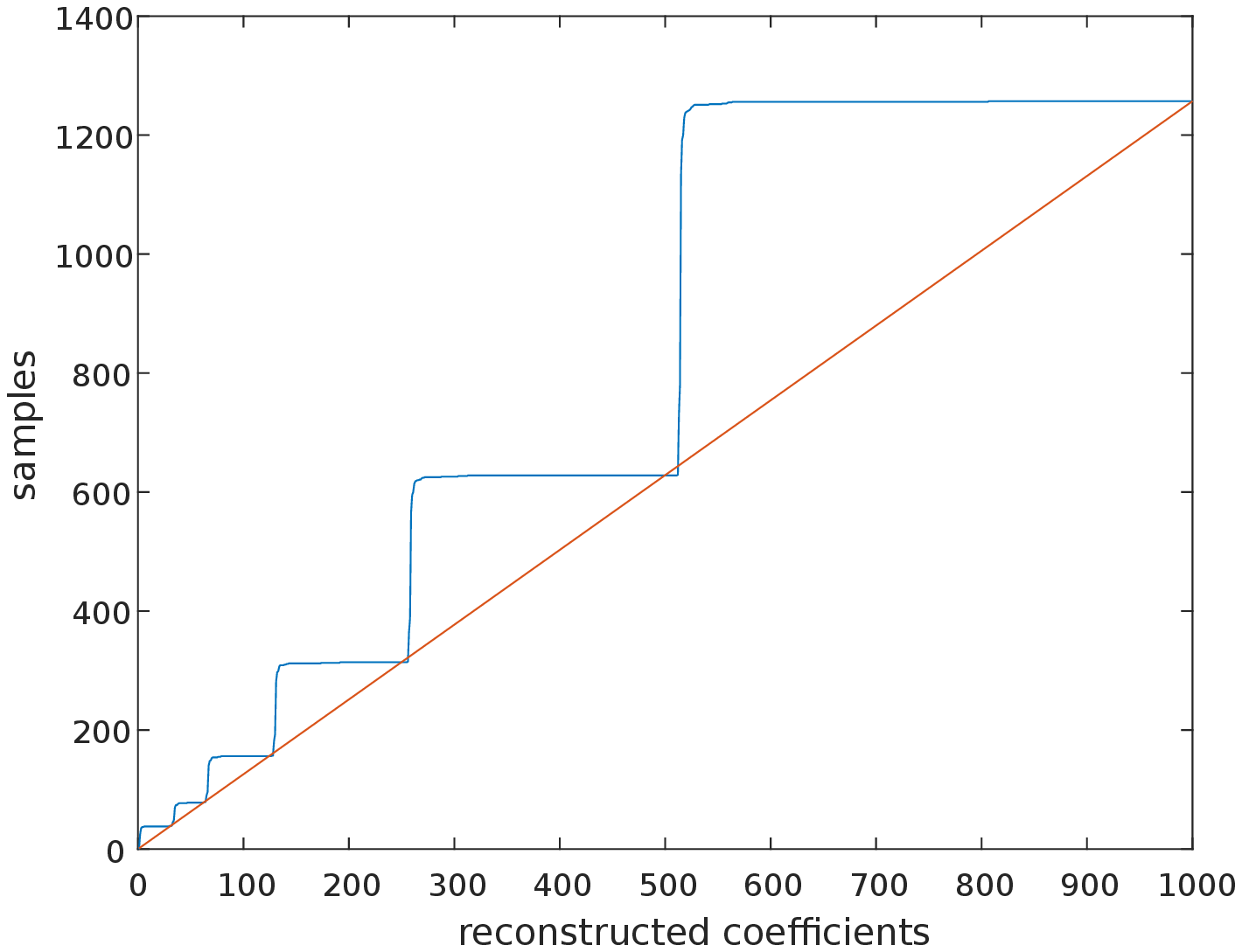}
		\label{fig:DB8Theta5}}
	\quad
	\subfloat[$\Theta(N;2)$ for DB8 with $\frac {N_{max}}{M_{max}} = 2.003$]{
		\includegraphics[width=0.45\textwidth]{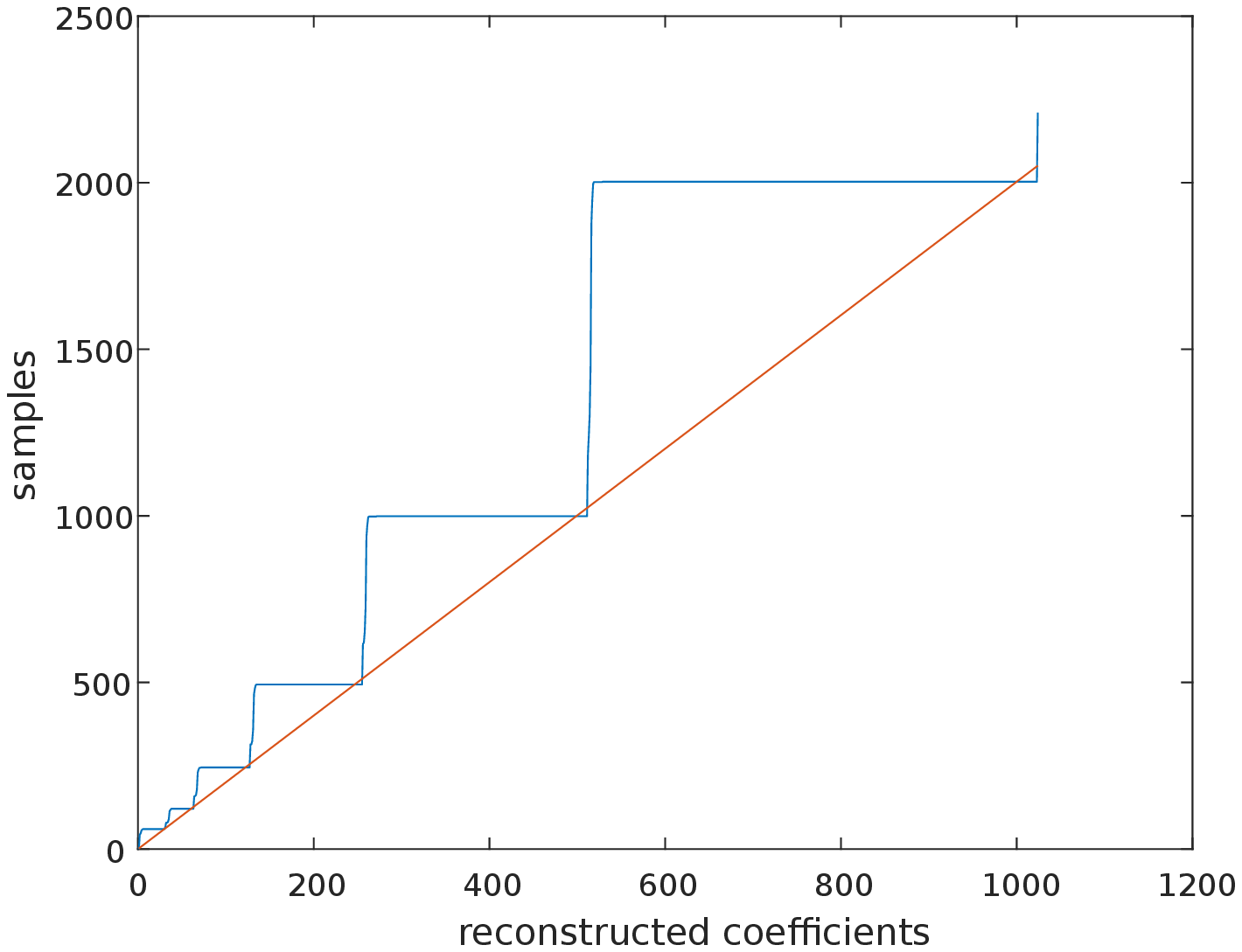}
		\label{fig:DB8Theta2}}
	\caption{Plots of the stable sampling rate (blue) for Daubechies Wavelet of order $2$ and $8$ for a threshold $\theta = 2$ and $5$ and the linear line with $N_{max}/M_{max}$ (orange). }
	\label{fig:StableSamplingRate}
\end{figure}

\begin{figure}[!htbp]
	\centering
	\subfloat[Original function $1$]{
		\includegraphics[width=0.45\textwidth]{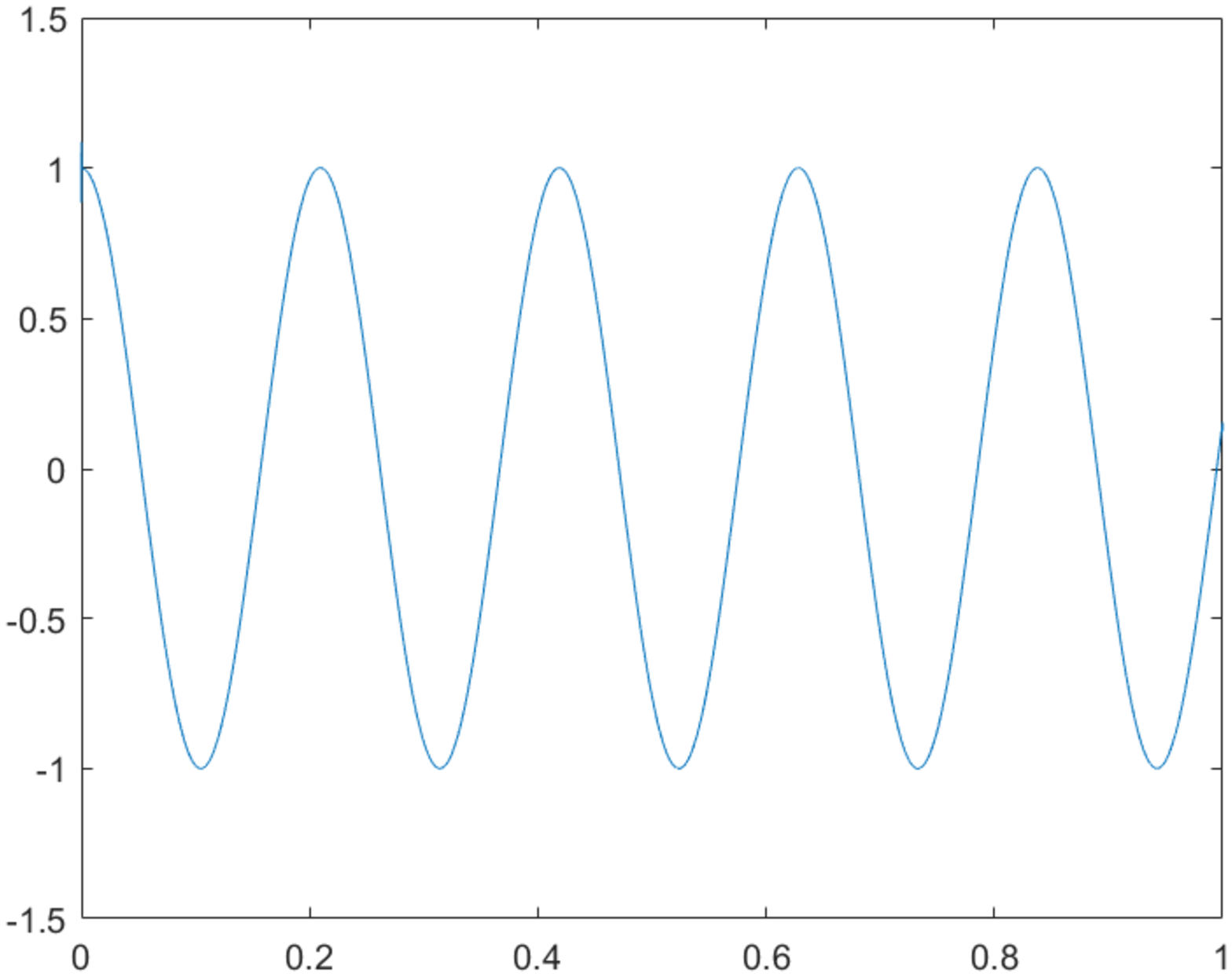}
		\label{fig:Orig}}
	\quad
	\subfloat[Reconstruction with Generalized Sampling of function $1$ with $64$ Wavelet coefficients from $77$ measurements]{
		\includegraphics[width=0.45\textwidth]{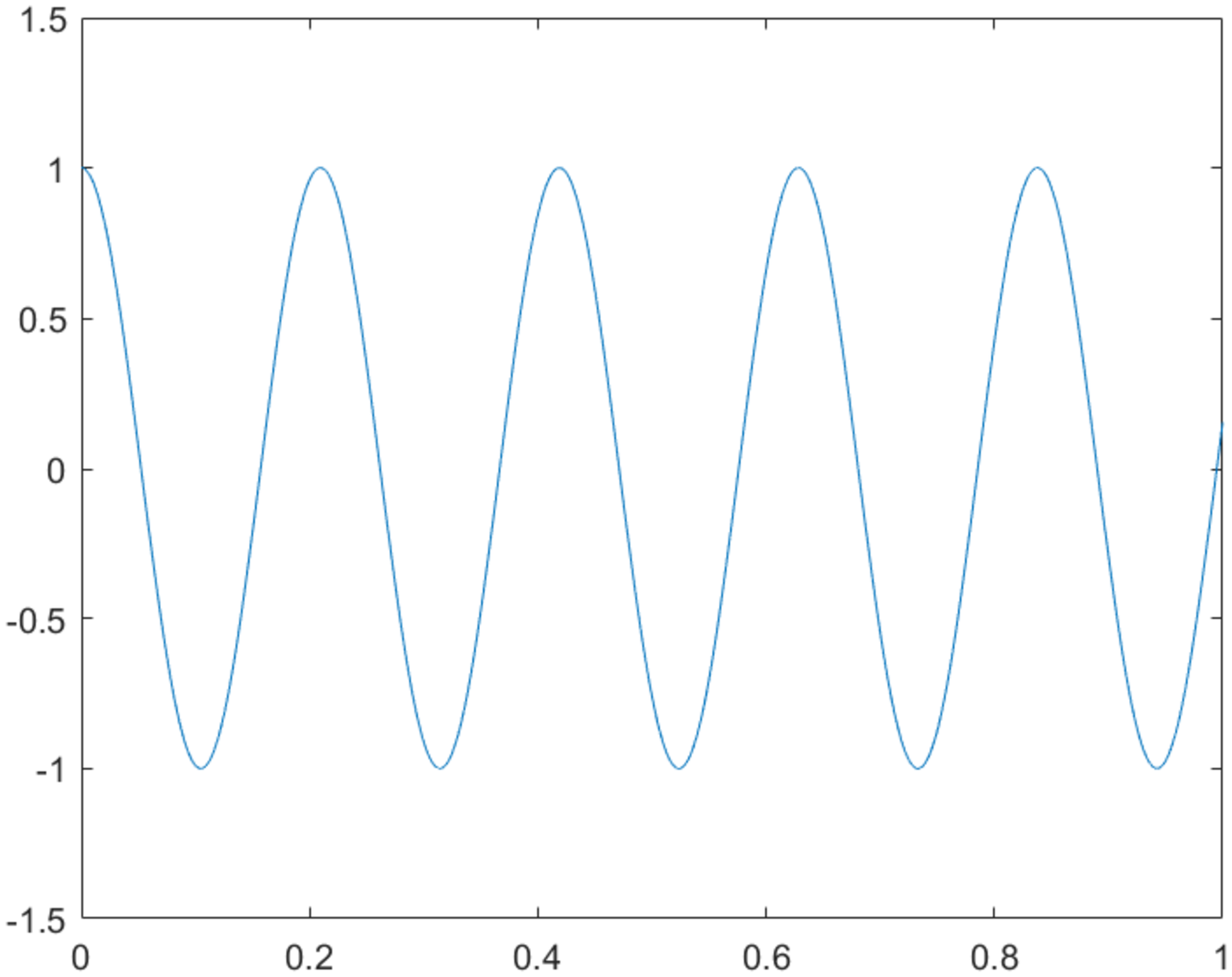}
		\label{fig:GS}}
	\quad
	\subfloat[Truncated Walsh series of function $1$ from $77$ measurements]{
		\includegraphics[width=0.45\textwidth]{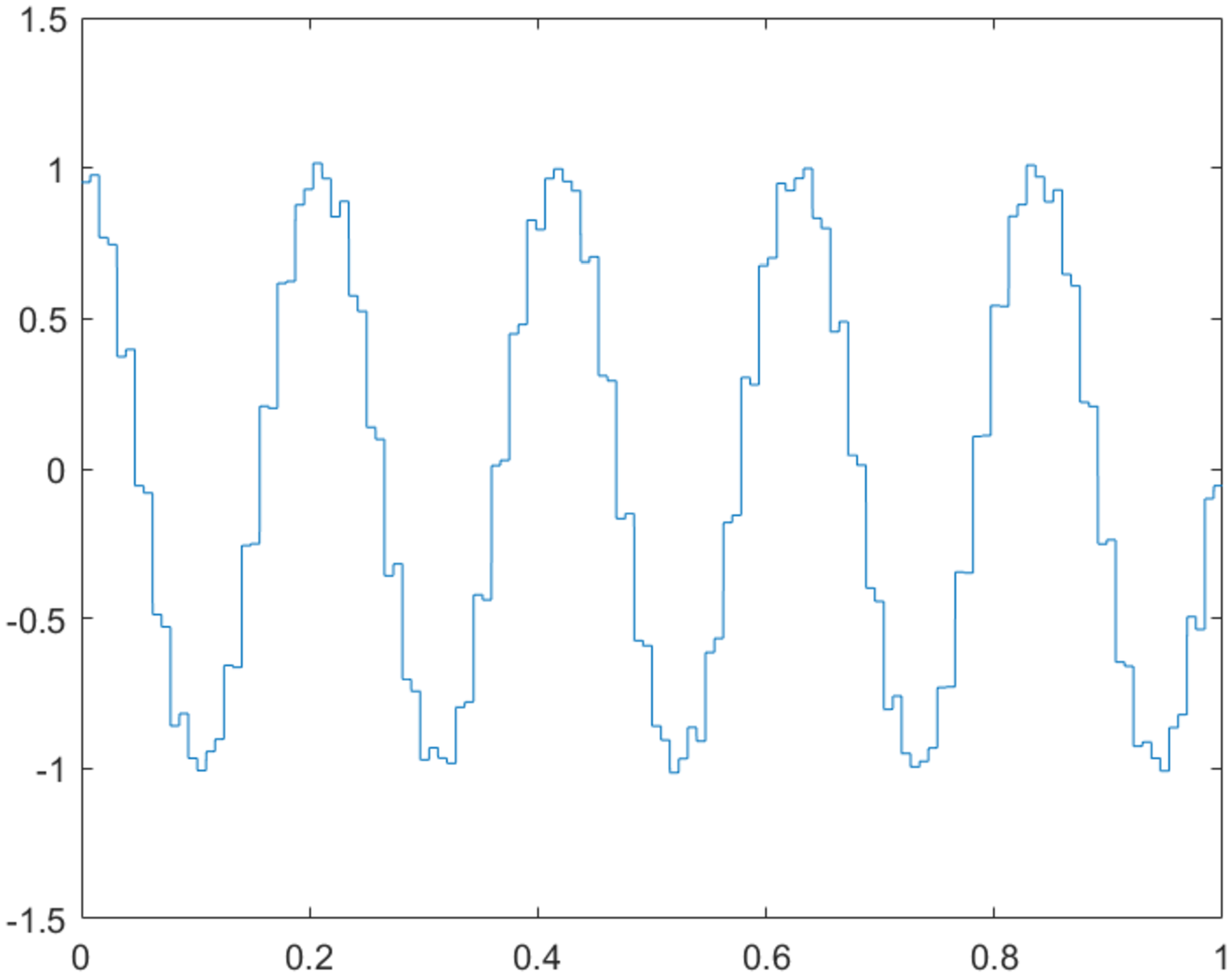}
		\label{fig:TW}}
	\quad	
		\subfloat[Original function $2$]{
			\includegraphics[width=0.45\textwidth]{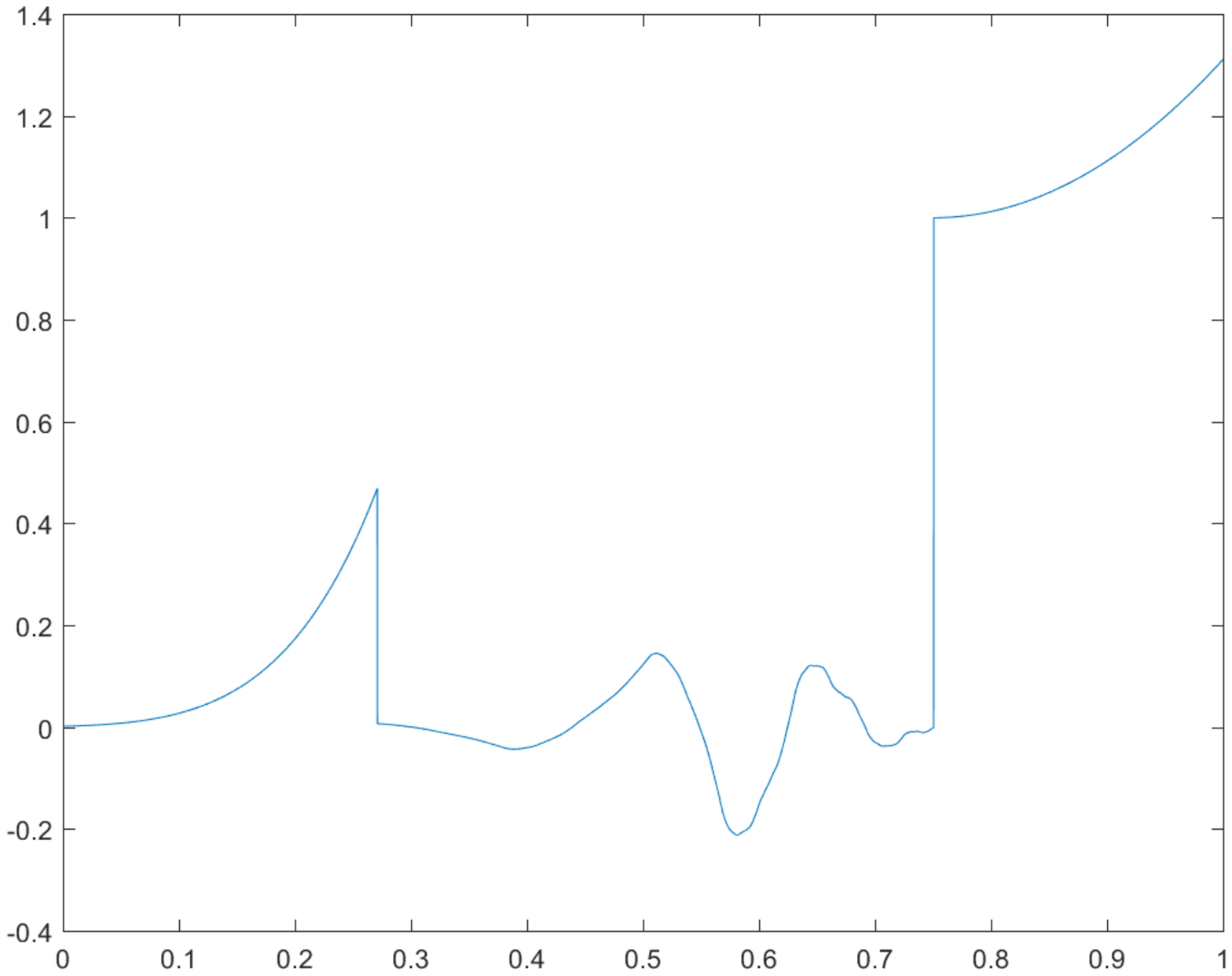}
			\label{fig:Orig2}}
		\quad
		\subfloat[Reconstruction with Generalized Sampling of function $2$ with $128$ Wavelet coefficients and $192$ measurements]{
			\includegraphics[width=0.45\textwidth]{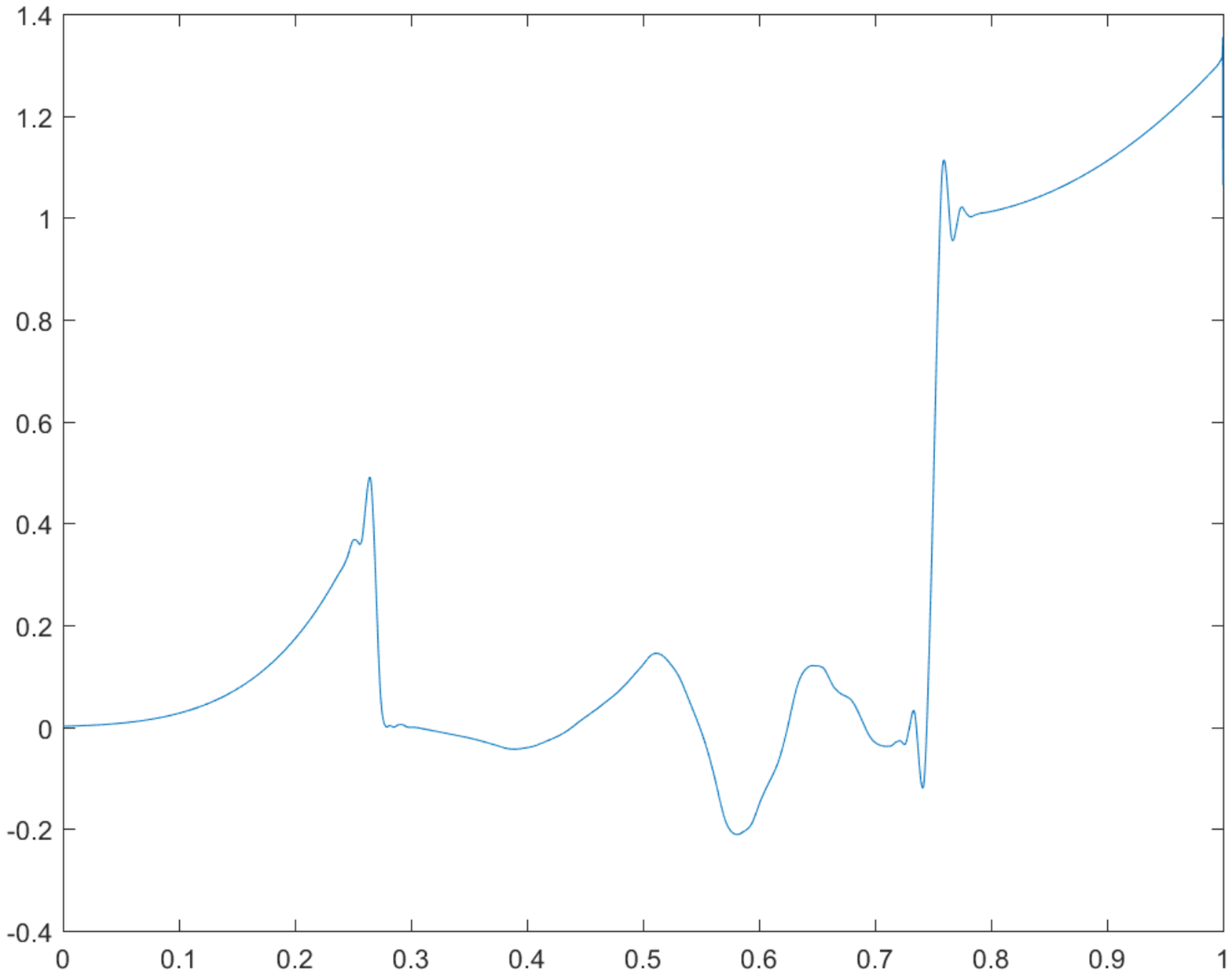}
			\label{fig:GS2}}
		\quad
		\subfloat[Truncated Walsh series from $192$ measurements]{
			\includegraphics[width=0.45\textwidth]{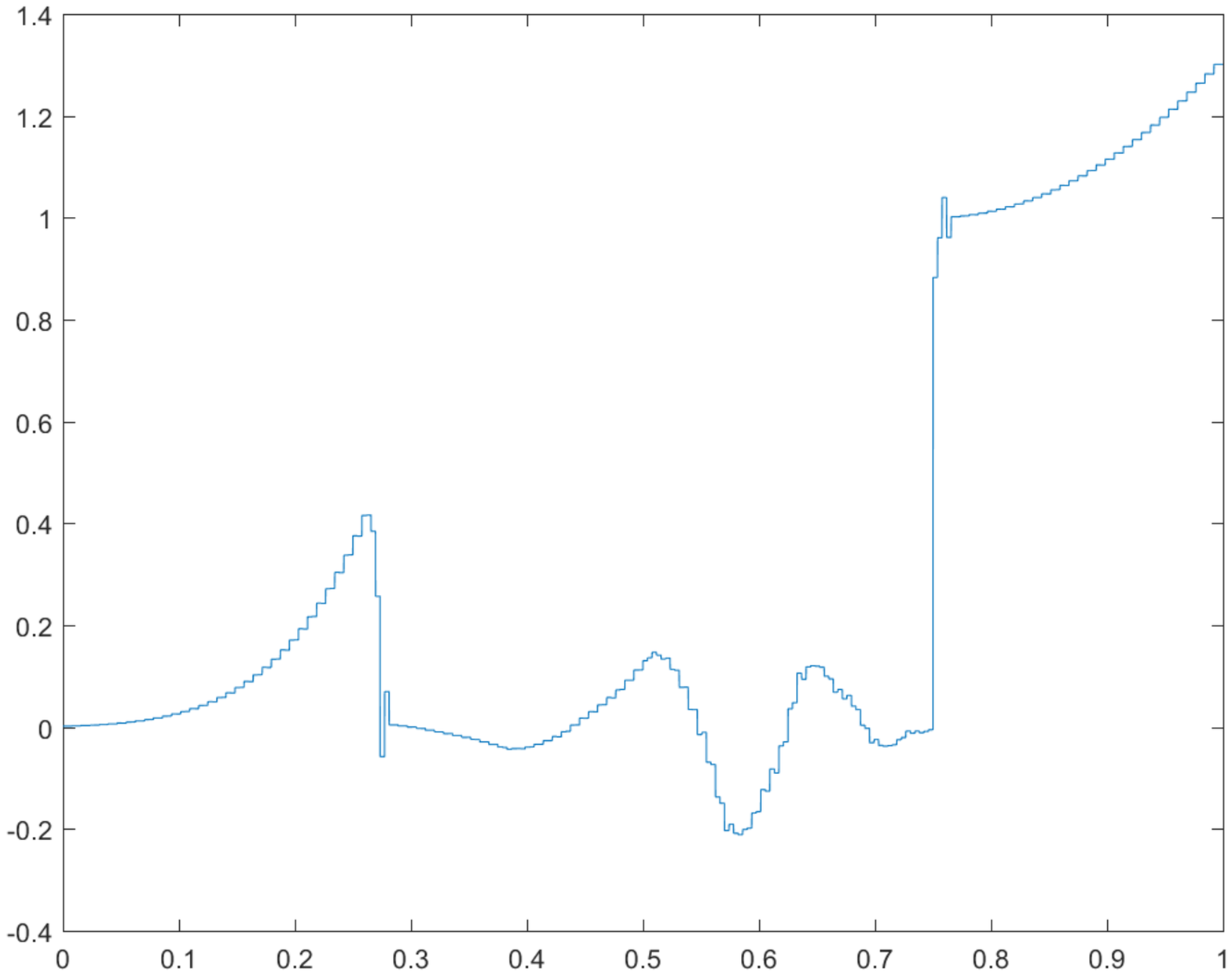}
			\label{fig:TW2}}
		\quad
		
	\caption{Reconstruction with Generalized Sampling and Daubechies $8$ Wavelets and the inverse Walsh.}
	\label{fig:GSvsTW}
\end{figure}

\begin{figure}
	\includegraphics[width=0.45\textwidth]{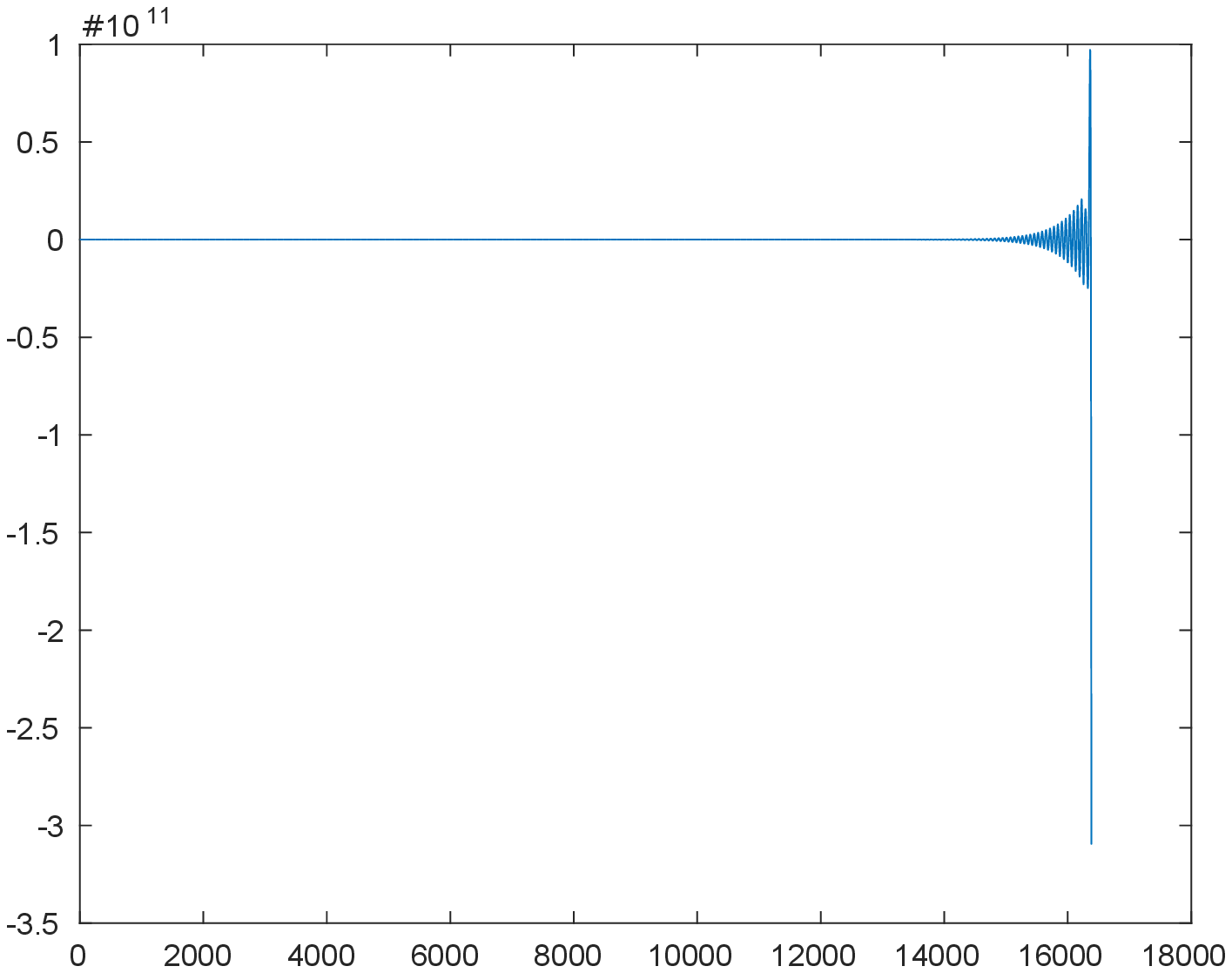}

	\caption{Reconstruction with Generalized Sampling below the Stable Sampling Rate with $512$ Walsh samples and Daubechies $8$ Wavelet coefficients}
		\label{fig:GSunderSR}
\end{figure}

In Figure \ref{fig:GSvsTW}, we demonstrate the reconstruction with generalized sampling. For this sake we consider two different functions on $[0,1]$. First, we look at the cosine function in Figure \ref{fig:Orig}, taking $77$ Walsh samples. In Figure \ref{fig:TW} the direct inversion is shown. It is clear that the reconstruction has a lot of block artefacts, whereas the reconstruction with generalized sampling of $64$ Daubechies $8$ wavelets has nearly no visible artefacts. The same artefacts can be seen in \ref{fig:TW2}. In this case $192$ Walsh samples were taken and $128$ wavelet coefficients were reconstructed. The artefacts with the direct Walsh inverse are much stronger than the common Gibbs phenomena for the Fourier case. Because of this, reconstructions with Walsh functions are not feasible in practice. They are also the reason why one does not use Haar wavelets as they obey the same block artefacts. This underlines the need of a reconstruction technique that refers the data from the sampling space to a much more appropriate reconstruction space, where the data is represented sparsely. In this case we get fewer artefacts. For completely continuous functions as in \ref{fig:GS} the reconstruction with generalized sampling has nearly no artefacts. In case of discontinuities as in the second function \ref{fig:Orig2} one gets some artefacts at the discontinuities as can be seen in \ref{fig:GS2}. Even so, the overall reconstruction quality is much better and the reconstruction still obeys the regularity properties of each part of the function. 

Nevertheless, it is important to take the stable sampling rate in mind. If one tries to reconstruct with fewer samples then needed, the reconstruction gets very unstable and one gets meaningless results. This can be seen in Figure \ref{fig:GSunderSR}, where a function on $[0,1]$ is reconstructed from $512$, which is much more than the $77$ and $192$ for the other functions.

\section{Conclusion}

We were able to investigate a very important part of the error estimate for different reconstruction methods. Moreover, we showed that binary measurements modelled by Walsh functions are well suited to reconstruct images with wavelets. This, together with the results in \cite{2DCase,linearity}, gives a broad knowledge about the accuracy and stability for two major applications of sampling theory, i.e. systems with Fourier samples and those with binary measurements.

\section*{Acknowledgement}

LT acknowledges support by the UK Engineering and Physical Sciences Research Council (EPSRC) grant EP/L016516/1 for the University of Cambridge Centre for Doctoral Training, the Cambridge Centre for Analysis. AH acknowledges support from a Royal Society University Research Fellowship.

\bibliographystyle{abbrv}

\end{document}